\documentclass[11pt,letterpaper]{amsart}
\usepackage[utf8]{inputenc}
\usepackage[T1]{fontenc}
\usepackage{ulem}

\usepackage{comment}

\usepackage{graphicx, xcolor}
\usepackage{tikz}
\usetikzlibrary{arrows,matrix,shapes.geometric}
\usepackage[all]{xy}
\usepackage{calligra,mathrsfs,amsthm,amscd, amsmath, amssymb}
\usepackage{float}
\usepackage{hyperref}
\hypersetup{colorlinks,allcolors=cyan}

\newtheorem{theo}{Theorem}[section]
\newtheorem{prop}[theo]{Proposition}
\newtheorem{lemm}[theo]{Lemma}
\newtheorem{cor}[theo]{Corollary}
\newtheorem{claim}[theo]{Claim}

\newtheorem{conj}[theo]{Conjecture}

\numberwithin{equation}{section}

\theoremstyle{definition}
\newtheorem{defi}[theo]{Definition}
\newtheorem{ex}[theo]{Example}
\newtheorem{step}{Step}
\newtheorem{setting}[theo]{Setting}

\theoremstyle{remark}
\newtheorem{rem}[theo]{Remark}

\newcommand{\codim}[0]{\operatorname{codim}}
\newcommand{\Aut}[0]{\operatorname{Aut}}
\newcommand{\pr}[0]{\operatorname{pr}}
\newcommand{\Image}[0]{\operatorname{Im}}
\newcommand{\Supp}[0]{\operatorname{Supp}}
\newcommand{\End}[0]{\operatorname{End}}

\newcommand{\QQ}{\mathbb{Q}}
\newcommand{\Exc}{\mathrm{Exc}}

\newcommand{\ddbar}{dd^c}

\newcommand{\dbar}{\overline{\partial}}
\newcommand{\e}{\varepsilon}

\newcommand{\OX}{\mathcal{O}}

\newcommand{\Unv}[1]{{#1}^{\rm{univ}}}

\newcommand{\ilim}[1][]{\mathop{\varinjlim}\limits_{#1}}

\def \ZZ {\mathbb Z} %
\def \CC {\mathbb C} %
\def \QQ {\mathbb Q} %
\def \PP {\mathbb P}  %

\newcommand{\inv}{^{-1}} 
\def \ZZ {\mathbb Z} %
\def \scrJ {\mathscr J} %

\def \calE {\mathcal E} %
\def \calF {\mathcal F} %
\def \calG {\mathcal G} %
\def \calH {\mathcal H} %
\def \calI {\mathcal I} %
\def \calV {\mathcal V} %
\def \calW {\mathcal W} %

\newcommand{\SheafHom[1]}{\mathscr{H}\!\!\!\text{\calligra om}_{\,{#1}}}

\DeclareMathOperator{\id}{id} %
\DeclareMathOperator{\Coh}{H} %
\DeclareMathOperator{\dimcoh}{h} %
\DeclareMathOperator{\Pic}{Pic} %
\DeclareMathOperator{\rank}{rk} %
\DeclareMathOperator{\reg}{reg} %
\DeclareMathOperator{\sing}{sing} %
\DeclareMathOperator{\Sym}{Sym} %
\DeclareMathOperator{\volume}{vol}

\DeclareMathOperator{\Exceptional}{Exc}

\newcommand{\terminal}{\text{\rm\tiny term}}
\newcommand{\qf}{\text{\rm\tiny qf}}

\DeclareMathOperator{\free}{f} %
\DeclareMathOperator{\base}{b} %

\DeclareMathOperator{\GL}{GL}
\DeclareMathOperator{\PGL}{PGL}

\let\Gamma\varGamma
\let\Delta\varDelta
\let\Theta\varTheta
\let\Lambda\varLambda
\let\Xi\varXi
\let\Pi\varPi
\let\Sigma\varSigma
\let\Upsilon\varUpsilon
\let\Phi\varPhi
\let\Psi\varPsi
\let\Omega\varOmega

\begin{document}

\date{\today, version 0.01}

\title[]{Structure theorem for projective klt pairs \\with nef anti-canonical divisor}

\author{Shin-ichi MATSUMURA}
\address{Mathematical Institute, Tohoku University, 
6-3, Aramaki Aza-Aoba, Aoba-ku, Sendai 980-8578, Japan.}
\email{{\tt mshinichi-math@tohoku.ac.jp}}
\email{{\tt mshinichi0@gmail.com}}

\author{Juanyong WANG}
\address{Academy of Mathematics and Systems Science, Chinese Academy of Sciences, 
55 East Road Zhongguancun, 100190 Beijing, China} 
\email{{\tt juanyong.wang@amss.ac.cn}}
\email{{\tt serredeciel@gmail.com}}

\renewcommand{\subjclassname}{%
\textup{2010} Mathematics Subject Classification}


\maketitle

\begin{abstract}
In this paper, we establish a structure theorem for a projective klt pair $(X,\Delta)$ with nef anti-log canonical divisor. 
Specifically, we prove that, up to replacing $X$ with a finite quasi-\'etale cover, the variety $X$ 
admits a locally trivial rationally connected fibration onto a projective klt variety with numerically trivial canonical divisor. 
As an application, we extend the Beauville-Bogomolov decomposition to projective klt Calabi-Yau pairs, by showing that klt Calabi-Yau pairs, which naturally appear as an outcome of the Log Minimal Model Program, are decomposed into building block varieties, namely rationally connected varieties and Calabi-Yau varieties. 


\end{abstract}          
 
\tableofcontents

\section{Introduction}\label{sec-intro}

\subsection{Structure of varieties with  `semi-positive' curvature}
We work over the complex number field throughout this paper. 
A central problem in birational geometry is revealing the structures of fibrations naturally associated with varieties 
so that the projective varieties can be decomposed into basic building blocks, 
namely  Fano varieties, Calabi-Yau varieties, and (log) canonical models. 
The abundance conjecture predicts that varieties with `semi-negative' curvature can be revealed by the Iitaka-Kodaira fibrations via the Minimal Model Program (MMP). 
Meanwhile, the pioneering studies \cite{Mori79, SY80}  showed that  varieties with `semi-positive' curvature have 
a certain rigidity and are closely related to the geometry of rational curves. 
Based on this philosophy, several structure theorems have been established for Albanese maps or maximal rationally connected (MRC) fibrations of varieties with various semi-positivity conditions, including compact K\"ahler manifolds with semi-positive holomorphic bisectional curvature \cite{HSW81, MZ86, Mok88}; projective manifolds with semi-positive holomorphic sectional curvature \cite{Yang16, Yang18, HW20, Mat18c}; compact K\"ahler manifolds with nef tangent bundle \cite{CP91, DPS94}; projective manifolds with pseudo-effective tangent bundle \cite{HIM19}; and projective manifolds with nef anti-canonical divisor \cite{Cao19, CH19, DLB20}.
In particular, projective manifolds with nef anti-canonical divisor cover a large range of varieties with `semi-positive' curvature, and the study  of anti-canonical divisors is more natural than that of tangent bundles 
from the perspective of birational geometry. 

In this paper, we study  projective varieties with nef anti-canonical divisor. The class of such varieties includes Fano manifolds and Calabi-Yau manifolds as extreme cases, where Mori's bend and break and the Beauville-Bogomolov decomposition have respectively played a decisive role in revealing their structures. 
Furthermore, on the basis of the previous works \cite{Paun97, Zhang96, Zhang05}, Cao-H\"oring established a structure theorem for smooth projective varieties with nef anti-canonical divisor, which interpolates between `Fano-like' manifolds (rationally connected manifolds) and  Calabi-Yau manifolds.

\begin{itemize}
\item[$\bullet$] (Mori's bend and break, \cite{Mori79, KoMM92, Cam92}). 
Fano manifolds (i.e.,\,smooth projective varieties with ample anti-canonical divisor) are rationally connected (i.e.,\,any two points can be connected by a rational curve). 
\item[$\bullet$] (Beauville-Bogomolov decomposition, \cite{Bea83}).
Calabi-Yau manifolds (i.e.,\,compact K\"ahler manifolds with numerically trivial canonical divisor) 
are decomposed, up to a finite \'etale cover, into the product of abelian varieties, strict Calabi-Yau manifolds, and 
holomorphic symplectic (hyperk\"ahler) manifolds. 
\item[$\bullet$] (Cao-H\"oring's structure theorem, \cite{Cao19, CH19}). 
Projective manifolds with nef anti-canonical divisor are constructed by rationally connected manifolds and Calabi-Yau manifolds. 
\end{itemize}

From the perspective of the MMP, it is more natural and of great importance to study the structure theorem not only for smooth varieties but also for varieties with klt singularities. 
In this direction,  the works \cite{Zhang06, HM07} generalized the rational connectedness of Fano manifolds to varieties of weak Fano type;  
that is, if $(X, \Delta)$ is a weak Fano pair (i.e.,\,a klt pair with nef and  big anti-log canonical divisor), then $X$ is rationally connected. 
Furthermore, the successive works \cite{GKP16a, Druel18a, GGK19, HP19, CP19, Cam20} generalized the Beauville-Bogomolov decomposition to klt Calabi-Yau varieties, which we refer to as {\textit{the singular Beauville-Bogomolov decomposition}} in this paper.
The most important  problem remaining in this field is to establish a structure theorem  for \textit{klt pairs with nef anti-log canonical divisor}, in the form of extending Cao-H\"oring's structure theorem, Zhang-Hacon-$\mathrm{M^{c}}$Kernan's result, and the singular Beauville-Bogomolov decomposition.

\subsection{Projective klt pairs with nef anti-log canonical divisor}

In this paper, we establish a structure theorem 
for projective klt pairs  with nef anti-log canonical divisor (see {\hyperref[thm-main]{Theorem \ref*{thm-main}}}), 
which  naturally generalizes Cao-H\"oring's structure theorem for smooth projective varieties 
and Hacon-$\mathrm{M^{c}}$Kernan's result for  varieties of weak Fano type. 
Moreover, our structure theorem reduces some problems on the structure of klt pairs with nef anti-log canonical divisor 
to the singular Beauville-Bogomolov decomposition for klt Calabi-Yau varieties.

\begin{theo}\label{thm-main}
Let $(X, \Delta)$ be a projective klt pair with the nef anti-log canonical divisor $-(K_X+\Delta)$. 
Then, there exists a finite quasi-\'etale  cover $\nu \colon X' \to X$ 
satisfying the following properties$:$
\begin{enumerate}
\item[$(1)$] $X'$ admits a holomorphic MRC fibration $\psi\colon X' \to Y$.   

\item[$(2)$] $Y$ is a projective klt variety with numerically trivial canonical divisor. 

\item[$(3)$] $\psi\colon X' \to Y$ is a locally constant fibration with respect to the pair $(X',   \Delta')$, 
where $\Delta'$ is the Weil $\QQ$-divisor defined by the pullback $\Delta':=\nu^* \Delta$. 
In particular, the fibration $\psi\colon X' \to Y$ is locally trivial with respect to $(X', \Delta')$, 
i.e.,\,for a sufficiently small open set $B \subset Y$, there exists an isomorphism
$$(\psi^{-1} (B), \Delta' ) \simeq B \times (F, \Delta' _{F})$$
over $B \subset Y$, 
where $F$ is the fiber of $\psi$ and $\Delta' _{F}:=\Delta' |_{F}$.
\end{enumerate} 
\end{theo}

Here {\it locally constant fibrations} are defined  in {\hyperref[def-constant]{Definition \ref*{def-constant}}}. 
The formulation of \hyperref[thm-main]{Theorem \ref*{thm-main}} using locally constant fibrations  is essentially important. 
This viewpoint of locally constant fibrations has not explicitly appeared in previous studies yet, but is often important in applications. 
For the time being, this fibration can be regarded as a locally trivial fibration while keeping in mind that the local constancy is a much stronger condition than the local triviality.
The structure theorems established in the previous studies do not require quasi-\'etale covers to be considered, 
but taking an appropriate quasi-\'etale cover is essential in the above theorem. 
Indeed, {\hyperref[thm-main]{Theorem \ref*{thm-main}}} does not hold without taking finite quasi-\'etale covers 
(see {\hyperref[rem-ex]{Remark \ref{rem-ex}}} for such an example).

By the singular Beauville-Bogomolov decomposition, the base variety $Y$ in {\hyperref[thm-main]{Theorem \ref*{thm-main}}}  has a finite quasi-\'etale cover $Y' \to Y$ such that $Y' $ is decomposed into the product of abelian varieties, strict Calabi-Yau varieties, and singular holomorphic symplectic varieties. Hence, by replacing $X'$ with the fiber product of $X' \times_{Y } Y'$, we deduce the following corollary: 

\begin{cor}\label{cor-main1}
Let $(X, \Delta)$ be a projective klt pair with the nef anti-log canonical divisor $-(K_X+\Delta)$. 
Then, there exists a finite quasi-\'etale cover $\nu\colon X' \to X$ 
satisfying the properties listed in {\hyperref[thm-main]{Theorem \ref*{thm-main}}} 
such that the base variety $Y$  of $\psi\colon X' \to Y$ 
is decomposed into the product
$$
Y \simeq A \times \prod_{i} Y_{i} \times \prod_{j} Z_{j}
$$
of an abelian  variety $A$,  strict Calabi-Yau varieties $Y_{i}$, and singular holomorphic symplectic varieties $Z_{j}$.  
\end{cor}

Moreover, for klt Calabi-Yau pairs, 
we strengthen {\hyperref[thm-main]{Theorem \ref*{thm-main}}} to a splitting theorem (see {\hyperref[cor-main2]{Theorem \ref*{cor-main2}}})
by combining {\hyperref[cor-main1]{Corollary \ref*{cor-main1}}} with 
\cite[Proposition 4.4, Theorem 4.7]{Amb05} and \cite[Lemma 4.6]{Druel18a}. 
Assuming the abundance conjecture, klt Calabi-Yau pairs naturally appear as an outcome of the Log MMP, as log minimal models of klt pairs of Kodaira dimension zero and as general fibers of the Iitaka-Kodaira fibration of log minimal models of positive Kodaira dimension. 
{\hyperref[cor-main2]{Theorem \ref*{cor-main2}}} shows that 
klt Calabi-Yau pairs can be further decomposed into building blocks comprising  rationally connected varieties and Calabi-Yau varieties.

\begin{theo}\label{cor-main2}
Let $(X, \Delta)$ be a projective klt pair with numerically trivial log canonical divisor $K_X+\Delta$. 
Then, there exists a finite quasi-\'etale  cover $\nu\colon X' \to X$ 
such that $(X', \Delta')$ is decomposed into the product 
$$
(X', \Delta') \simeq (F, \Delta'_{F}) \times A \times \prod_{i} Y_{i} \times \prod_{j} Z_{j}, 
$$
where $\Delta':=\nu^\ast\Delta$, $F$ is the fiber of the MRC fibration of $X'$ $($in particular $F$ is rationally connected$)$ with $\Delta'_F:=\Delta'|_F$, $A$ is an abelian  variety, $Y_i$ is a  strictly Calabi-Yau variety, and $Z_i$ is a singular holomorphic symplectic variety 
as in {\hyperref[thm-main]{Theorem \ref*{thm-main}}} and {\hyperref[cor-main1]{Corollary \ref*{cor-main1}}}.
\end{theo}

{\hyperref[cor-main2]{Theorem \ref*{cor-main2}}} reveals the structure of some building blocks in the framework of the Log MMP, although {\hyperref[thm-main]{Theorem \ref*{thm-main}}} itself is not a part of the Log MMP. Note that \hyperref[cor-main2]{Theorem \ref*{cor-main2}} does not hold  for lc pairs in general (see \cite[Example 6.2]{EIM20}).

\subsection{Open problems on fundamental groups and slope rationally connected  quotients}
\label{subsec-1.3}

In this subsection, we discuss several topics 
relating to the uniformization, (topological) fundamental groups, 
and slope rationally connected (sRC)  quotients of klt pairs with nef anti-log canonical divisor. 
This subsection does \textit{not} directly relate to the proof of our main result, 
but illustrates the difficulty and interest in handling klt singularities.

We first review the proof of {\hyperref[thm-main]{Theorem \ref*{thm-main}}} 
in the case where $X$ is a smooth projective variety with nef anti-canonical divisor. 
The proof in this case is divided into the following steps:

\begin{itemize}
\item[Step$ (1)$] (Study of fundamental groups, \cite{Paun97, Paun17}).
The fundamental group $\pi_{1}(X)$ is shown to be virtually abelian 
by the theory of Cheeger-Colding \cite{CC96}.

\item[Step$(2)$] (Study of Albanese maps, \cite{Cao19}).  
The Albanese map of $X$ is shown to satisfy the desired structure theorem. 
Then, together with Step (1),  the problem is reduced to the case of $X$ being simply connected. 

\item[Step$(3)$] (Study of MRC fibrations, \cite{CH19}). 
The desired structure theorem for MRC fibrations is completed 
in the case of $X$ being simply connected (and thus in the general case by Step (2)). 
\end{itemize}
Note that the previous works \cite{Zhang96, Zhang05} played an important role in establishing Steps (2) and  (3). 

After Campana-Cao-Matsumura \cite{CCM19} initiated the study of klt pairs $(X,\Delta)$, 
the second author \cite{Wang20} revealed  that 
the difficulty in adopting the above strategy arises in Step (1); 
precisely, {\hyperref[thm-main]{Theorem \ref*{thm-main}}} as well as the uniformization theorem 
can be deduced if the fundamental group  $\pi_{1}(X_{\reg})$ of the regular locus $X_{\reg}$ has  polynomial growth. 

\begin{conj}[{\cite[Conjecture 1.2]{Wang20}}]
\label{conj2}
Let $(X,\Delta)$ be a projective klt pair with nef anti-log canonical divisor $-(K_X+\Delta)$. 
Then, the fundamental group $\pi_{1}(X_{\reg})$ has  polynomial growth. 
\end{conj}

{\hyperref[conj2]{Conjecture \ref*{conj2}}} is partially solved 
for the orbifold fundamental group $\pi_{1}^{{\rm{orb}}}(X, \Delta)$ of an orbifold pair $(X, \Delta)$,
but it seems quite difficult to solve {\hyperref[conj2]{Conjecture \ref*{conj2}}} in the general case. 
In fact, it is even unclear whether the argument in \cite{Paun97, Paun17} works for klt pairs $(X, \Delta)$ with smooth $X$. 
Since {\hyperref[conj2]{Conjecture \ref*{conj2}}} for smooth  projective varieties  is a starting point of \cite{CH19} as explained above, 
we need to take another strategy for the proof of {\hyperref[thm-main]{Theorem \ref*{thm-main}}}. 
Our strategy  displayed in this paper 
is independent of the results on fundamental groups and is new even in the smooth case. 

{\hyperref[thm-main]{Theorem \ref*{thm-main}}} can be applied 
to prove the uniformization theorem of klt pairs $(X, \Delta)$ with nef anti-log canonical divisor. 
Indeed, {\hyperref[thm-main]{Theorem \ref*{thm-main}}}  reduces the uniformization problem for $(X, \Delta)$ to the following conjecture, 
which is formulated only for the fundamental group of Calabi-Yau varieties (and not for the regular locus),  
and thus more straightforward than {\hyperref[conj2]{Conjecture \ref*{conj2}}}.

\begin{conj}
\label{coj-main}
Let $Y$ be a projective klt variety with numerically trivial canonical divisor.
If $Y$ has vanishing augmented irregularity, 
then $\pi_{1}(Y)$ is finite. 
\end{conj}

{\hyperref[thm-main]{Theorem \ref*{thm-main}}}  opens the study of sRC quotients.
The MRC fibration of $X$ has the drawback of not containing any information of the boundary divisor $\Delta$, 
whereas the sRC quotient of $(X, \Delta)$, introduced in \cite{Cam16}, is a natural generalization of the MRC fibration 
that considers  the boundary divisor $\Delta$.
It would be an attractive problem to establish a structure theorem of  sRC quotients by comparing MRC fibrations of $X$ to sRC quotients of $(X, \Delta)$. 
Based on \cite[Conjecture 1.5]{CCM19}, we pose the following conjecture:
\begin{conj}
Let $(X,\Delta)$ be a klt pair with nef anti-log canonical divisor $-(K_X+\Delta)$. 
Then, up to a finite quasi-\'etale cover, the sRC quotient $(X, \Delta) \to (Z, \Delta_{Z})$ is a locally constant fibration 
with numerically trivial $K_{Z}+\Delta_{Z}$. 
Moreover, the MRC fibration $ X \to Y$ in {\hyperref[thm-main]{Theorem \ref*{thm-main}}} factorizes through $(X, \Delta) \to (Z, \Delta_{Z})$. 
\end{conj}

\subsection*{Organization of the paper}
This paper is organized as follows:  
In {\hyperref[sec-pre]{Section \ref*{sec-pre}}}, we recall some definitions and preliminary results,  
especially the definitions of {\it MRC fibrations}, {\it locally constant fibrations}, {\it singular Hermitian metrics}, and {\it finite quasi-\'etale covers}.  
Unlike the previous works for smooth projective varieties, 
we need quasi-\'etale covers in studying klt pairs for {\hyperref[thm-main]{Theorem \ref*{thm-main}}}
(see {\hyperref[rem-ex]{Remark \ref{rem-ex}}}). 
This requires non-trivial and technical discussions in Sections \ref{sec-flat} and \ref{sec-MRC}. 
In {\hyperref[sec-flat]{Section \ref*{sec-flat}}}, 
we prove a technical result (see {\hyperref[thm-flat]{Theorem \ref*{thm-flat}}}), 
which asserts that, up to replacing $X$ with a quasi-\'etale cover and a $\QQ$-factorialization, 
appropriately defined `direct image sheaves $\psi_*(mL)$' are flat for a positive integer $m \in \mathbb{Z}_{>0}$, 
where $\psi \colon X \dashrightarrow Y$ is an MRC fibration of $X$ and 
$L$ is a $\psi$-ample line bundle on $X$. 
The proof of this flatness clarifies 
the reason why we need a finite quasi-\'etale cover of $X$ in {\hyperref[thm-main]{Theorem \ref*{thm-main}}}. 
In {\hyperref[sec-MRC]{Section \ref*{sec-MRC}}} we deduce the desired structure for $(X,\Delta)$. By combining the flatness result above with the results of (singular) foliations established by Stéphane Druel, we can prove the structure theorem in a birational sense, and the main difficulty and the major contribution of this section is to deduce the main theorem from such a birational structure result.  
Let us remark that in the discussion in {\hyperref[sec-MRC]{Section \ref*{sec-MRC}}}, 
the formulation of \hyperref[thm-main]{Theorem \ref*{thm-main}} using locally constant fibrations  plays a crucial role. 
This viewpoint is important not only for the proof but also for applications of the theorem.

\subsection*{Acknowledgments} 
The authors wish to express their thanks to Professors 
S\'ebastien Boucksom, 
Fr{\'e}d{\'e}ric Campana, 
St\'ephane Druel, 
Andreas H\"oring, 
and J\'anos Koll\'ar 
for helpful comments and suggestions. 
The authors are deeply indebted to an anonymous referee for pointing out {\hyperref[thm-mqe]{Theorem \ref{thm-mqe}}}, 
which simplified the discussion in {\hyperref[sec-flat]{Section \ref*{sec-flat}}} of the draft version and 
clarified the essence of the proof.
 The first author  is grateful to 
Professor  Yoshinori Gongyo for a discussion on quasi-abelian varieties and 
for answering questions, 
and to Professors Shigefumi Mori, Yuji Odaka, and Ken-ichi Yoshikawa 
for providing him with an opportunity to talk at the Algebraic Geometry Seminar at Kyoto University, 
during which one question drew his attention to the case of singular varieties. 
He was supported 
by Grant-in-Aid for Young Scientists (A) $\sharp$17H04821, 
Grant-in-Aid for Scientific Research (B) $\sharp$ 21H00976, 
and Fostering Joint International Research (A) $\sharp$19KK0342 from JSPS. 
The second author is grateful to Professors St\'ephane Druel and Andreas H\"oring for kindly answering his questions on foliations. 
He was partially supported by the French ANR Research Project ``GRACK'' and the National Natural Science Foundation of China (Grant No. 12288201), and is partially supported by the National Key R\&D Program of China (Grant No. 2021YFA1003100).

\section{Preliminary results}\label{sec-pre}

\subsection{Notation and conventions}\label{subsec-notation}

Throughout this paper, 
all the  varieties and morphisms are defined over the complex number field, 
and analytic varieties denote irreducible and reduced complex analytic spaces. 
We interchangeably use the terms 
`locally free sheaves' and `vector bundles,'  
and further  the terms `Cartier divisors,' 
`invertible sheaves,'
and `line bundles.'
For example, a Cartier divisor $D$ is  regarded as a line bundle, 
and Hermitian metrics on $D$ denote those on the line bundle associated with $D$
(where this convention differs from \cite{Wang20}).
We often use an additional notation for the tensor products of invertible sheaves 
$L$ and $M$ (e.g.,\,$-L=L^*$ and $L+M=L \otimes M$).

\subsection{MRC fibrations}
We recall the definitions and properties of MRC fibrations and rationally (chain) connected varieties. 

\begin{defi}\label{def-mrc}
Let $X$ and $Y$ be normal projective varieties. 
A dominant rational map $\psi \colon X \dashrightarrow Y$ is called an \textit{RC fibration} 
if it satisfies the following conditions: 
\begin{itemize}
\item[$\bullet$] $\psi\colon X \dashrightarrow Y$ is almost holomorphic 
(i.e.,\,the indeterminacy locus is not dominant over $Y$, and a general fiber is connected). 
\item[$\bullet$] A general fiber is rationally connected. 
\end{itemize}

\end{defi}

There exist maximal rationally chain connected (MRCC) fibrations for normal projective varieties by \cite{Cam92, KoMM92a}. 
Rationally chain connected varieties are not necessarily rationally connected, 
but the two notions coincide for potentially dlt projective varieties by \cite[Corollary 1.5]{HM07}. 
Hence, for a projective klt pair $(X, \Delta)$, the variety $X$ always admits an MRC fibration 
(i.e.,\,an RC fibration $\psi\colon X \dashrightarrow Y$ such that a very general rational curve in $X$ is contained in 
the fiber $X_{y}:=\psi^{-1}(y)$ at some point $y \in Y$.   
This definition of MRC fibrations can be rephrased as the condition that 
$K_{Y}$ is pseudo-effective when $Y$ is smooth by \cite{BDPP13, GHS03}. 

Rationally connected varieties with mild singularities are simply connected by \cite[Theorem 1.1]{Tak03} 
and have no $($non-trivial$)$  holomorphic differential forms. 
These properties are used in the proof.

\begin{theo}\label{thm-funda}
Let $X$ be a rationally connected and potentially klt projective variety. 
Then $X$ is simply connected and $H^{0}(X, \Omega_{X}^{p})=0$ for any $p>0$. 
\end{theo}

\subsection{Locally constant fibrations}
In this subsection, after we give the definition of locally constant fibrations, 
we explain the properties of the locally constant fibrations 
(e.g.,\,{\hyperref[lemma_local-const-lb-decomp]{Lemma \ref{lemma_local-const-lb-decomp}}}) and 
a criterion for a fiber space to be locally constant using the flatness of certain direct image sheaves 
(see {\hyperref[prop_flat-lcf]{Proposition \ref{prop_flat-lcf}}}).

\begin{defi}\label{def-constant}
Let $\phi\colon X \to Y$ be a fiber space between normal analytic varieties 
(i.e.,\,a proper surjective  morphism with connected fibers), 
and let $\Delta$ be a Weil $\QQ$-divisor on $X$. 

\smallskip
\noindent 
(1)
$\phi \colon X \to Y$ is called a {\textit{locally constant fibration with respect to the pair $(X, \Delta)$}}
if it satisfies the following conditions: 
\begin{itemize}
\item[$\bullet$] $\phi \colon X \to Y$ is an analytic fiber  bundle with the fiber $F$. 
\item[$\bullet$] Every component $\Delta_{i}$ of $\Delta$ is horizontal $($i.e.,\,$\phi(\Delta_{i})=Y$$)$. 
\item[$\bullet$]

There exists a Weil $\QQ$-divisor $\Delta_{F}$ on $F$ and 
a representation 
$
\rho \colon \pi_1(Y)\to\Aut(F)
$
of the fundamental group $\pi_1(Y)$ to 
the automorphism group $\Aut(F)$   such that 
\begin{itemize}
\item[$\bullet$] $\Delta_{F}$ is invariant under the action of $\pi_1(Y)$; 
\item[$\bullet$] $(X, \Delta) $ is isomorphic to the quotient $(\Unv{Y}\times F, \pr_2^\ast\Delta_{F})/\pi_1(Y)$ over $Y$. 
Here $\Unv{Y}$ is the universal cover of $Y$ 
and $\gamma \in \pi_1(Y)$ acts on $\Unv{Y}\times F$ 
by the diagonal action  
$$
\gamma\cdot(y,z):=(\gamma\cdot y, \rho(\gamma)(z)) \text{ for } (y,z) \in \Unv{Y}\times F. 
$$
\end{itemize}
\end{itemize}

\smallskip
\noindent 
(2) $\phi \colon X \to Y$ is simply said to be {\textit{locally constant}} 
if it satisfies the above conditions for $\Delta=0$.

\end{defi}

Note that, when $\phi \colon X \to Y$ is a locally constant fibration, 
the fiber product $X\times_Y\Unv Y$ is isomorphic to the product $\Unv Y\times F$ 
and the induced morphism $X\times_Y\Unv Y\to \Unv Y$ is equal to 
the first projection $\Unv Y\times F \to \Unv Y$ under this identification. 
We thus have the following diagram: 
\begin{equation}\label{locally-constant}
\begin{gathered}\xymatrix@C=40pt@R=30pt{
X  \ar[d]_{\phi} &  X\times_Y\Unv Y \cong \Unv{Y} \times F \ar[d]_{\pr_{1}} \ar[l]^{p_{X} \qquad  \qquad}  
\ar[r]^{\qquad \qquad \pr_{2}} & F  \\ 
Y   & \ar[l]^{p_{Y}} \Unv{Y}.  &  \\   
}\end{gathered}
\end{equation}

A locally constant fibration $\phi \colon X \to Y$ with respect to a pair $(X, \Delta)$ 
is locally trivial, i.e.,\,it is an analytic fiber  bundle with  fiber $F$, 
and there exists a Weil $\QQ$-divisor $\Delta_{F}$ on $F$ such that 
$$(\phi^{-1} (B), \Delta)\simeq  B \times (F, \Delta_{F})$$
over a sufficiently small open set $B \subset Y$. 
However, the converse implication does not  hold in general; 
e.g.,\,the projective space bundle $\mathbb{P} (E) \to Y$ of a locally free sheaf $E$ on $Y$ 
is always locally trivial, but not necessarily locally constant.

The projective space bundle  $\mathbb{P} (E) \to Y$ is a locally constant fibration  if $E$ is flat. 
Here, we recall that a vector bundle $E$ on a smooth variety $Y$ is said to be \textit{flat} if 
$E$ admits a flat connection or, equivalently, if $E$ arises from a representation $\pi_1(Y)\to\GL_r(\CC)$, 
where $r:=\rank E$ 
(see \cite[\S 1.2, pp.\,4-6]{Kob87}). 
Moreover, the category of flat vector bundles is equivalent to that of local systems (i.e.,\,locally constant sheaves) 
via the Riemann-Hilbert correspondence. Hence, we interchangeably use them  in the sequel.

The previous works \cite{Cao19, CH19} formulated structure theorems using locally trivial fibrations. 
Nevertheless, it is important to formulate 
\hyperref[thm-main]{Theorem \ref*{thm-main}} with locally constant fibrations  for several reasons; 
e.g.,\,a locally constant fibration induces a splitting of tangent sheaves   
(see \cite[Remark 2.2]{Wang20} and the proof of {\hyperref[thm_splitting]{Theorem \ref*{thm_splitting}}}); 
moreover, our proof needs the following lemma (see {\hyperref[ss_MRC_klt]{Subsection \ref*{ss_MRC_klt}}}), 
which does not hold for locally trivial fibrations in general.

\begin{lemm}
\label{lemma_local-const-lb-decomp}
Let $\phi \colon X\to Y$ be a projective locally constant fibration 
such that the fiber  $F$ has vanishing irregularity. 
Let $L$ be a line bundle  $X$. 
Then, there exist a line bundle $L_{F}$ on $F$ and 
a $\mathbb{Q}$-line bundle $($$\mathbb{Q}$-Cartier divisor$)$ $L_{Y}$ on $Y$  such that 
the pullback of $p_{X}^{*}(L  -\phi^* L_{Y}) \sim_{\QQ}\pr_{2}^* L_{F}$ 
and  the direct image sheaf $\phi_*(m(L-\phi^* L_{Y})) $ is a flat vector bundle on $Y$ 
$($or  the zero sheaf$)$ for any sufficiently divisible integer $m \in \mathbb{Z}_{>0}$.  
\end{lemm}
\begin{proof}
We use the notations in the diagram \eqref{locally-constant}. 
Let $\rho \colon \pi_1(Y)\to \Aut(F)$ be a representation  such that $X$ is the quotient of 
$\Unv Y\times F$ by the action of $\pi_1(Y)$. 
By assumption, the fiber $F$ is a projective variety with vanishing irregularity. 
Hence we have the decomposition of the Picard group 
$\Pic(\Unv{Y}\times F)\simeq\Pic(\Unv{Y})\times\Pic(F)$ 
by the analytic version of \cite[\S I\!I\!I.12, Exercise 12.6, p.\,292]{Har77}. 
In particular, there exists a line bundle $L_{\Unv{Y}}$ on $\Unv{Y}$ and a line bundle $L_F$ on $F$ such that 
$p_X^\ast L\simeq \pr_1^\ast\!L_{\Unv{Y}} +\pr_2^\ast\!L_F$. 
Note that $L_F$ is a $\rho$-equivariant line bundle, but not necessarily $\rho$-linearizable. 

By applying the above argument to a $\phi$-ample line bundle, 
we can find a $\rho$-equivariant ample line bundle on $F$. 
Thus, the Zariski closure $G$ of  $\Image (\rho)$ is a linear algebraic group. 
By \cite[Proposition 2.4]{KKLV89}, 
we can take an integer $m_{0} \in \mathbb{Z}_{>0}$ 
such that $mL_{F}$ is $G$-linearizable 
for any integer $m \in \mathbb{Z}_{>0}$ divisible by $m_{0}$. 
(Precisely speaking, the cited result \cite[Proposition 2.4]{KKLV89} assumes that $G$ is connected, 
but this assumption can be removed since $G$ has only finitely many connected components.) 
The line bundle $L_{0}$ defined by $L_{0}:=(\Unv{Y} \times m_{0}L_{F})/\pi_{1}(Y)$ 
satisfies that $p_{X}^* L_{0} \simeq \pr_{2}^{*} (m_{0}L_{F})$ and $(m_{0}L-L_{0})|_F \simeq\mathcal{O}_{F}$. 
Hence, we can find a line bundle $L'_{Y}$ on $Y$ such that $m_{0}L-L_{0}\sim \phi^{*}L'_{Y}$. 

We finally show that the line bundle $L_{F}$ and 
the $\mathbb{Q}$-line bundle $L_{Y}:=(1/m_{0})L_{Y}'$ satisfy the desired properties.
By construction, we can easily find that $p_{X}^{*}(L  -\phi^* L_{Y}) \sim_{\QQ} \pr_{2}^{*}L_{F}$.  
Furthermore, for any $m \in \mathbb{Z}$ divisible by $m_{0}$, we can see that 
$$  
{\pr_{1}}_{*}\big({p_{X}}^* (m(L -  \phi^* L_{Y}))\big)= {\pr_{1}}_* ({\pr_{2}}^* (mL_{F})) = \mathcal{O}_{\Unv Y} {\otimes}_\CC H^0\big(F, mL_{F}\big).
$$
Hence, by the flat base change theorem, we obtain  
\begin{equation}\label{m-flat}
p_{Y}^*\big( \phi_* (m(L - \phi^* L_{Y}) )  \big) \cong \mathcal{O}_{\Unv Y}{\otimes}_\CC H^0\big(F, mL_{F}\big). 
\end{equation}
The representation $\rho \colon \pi_1(Y)\to \Aut(F)$ linearly acts on $H^0\big(F, mL_{F}\big)$, 
and the quotient by this action coincides with $\phi_* (m(L - \phi^* L_{Y}))$. 
This implies that $\phi_*(m(L-\phi^* L_{Y})) $ is a flat vector bundle on $Y$. 
\end{proof}

For a locally constant fibration  $X \to Y$, 
 every line bundle on $X$ can be modified 
(up to the tensoring the pullback of a line bundle on $Y$) so that its direct image is a flat vector bundle 
by {\hyperref[lemma_local-const-lb-decomp]{Lemma \ref*{lemma_local-const-lb-decomp}}}.
Conversely, if there exists a relatively very  ample line bundle on $X$ whose direct image sheaf is flat, 
then $X \to Y$ is necessarily a locally constant fibration by the following proposition: 

\begin{prop}
\label{prop_flat-lcf}
Let $h \colon V\to W$ be a fiber space from a normal analytic variety $V$ to a complex manifold  $W$,   
and let $D$ be an effective Weil $\QQ$-divisor on $V$. 
Assume that $h$ is a flat projective morphism, and assume that there is an $h$-relatively very ample line bundle $L$ on $V$ such that 
\begin{itemize}
\item[\rm(1)] $E_m:=h_\ast(mL)$ is a local system for every $m\in\ZZ_{>0}$$;$ 
\item[\rm(2)] for every $m\in\ZZ_{>0}$, the natural morphism $\Sym^m\!E_1\to E_m$ is a morphism between local systems $($i.e.,\,compatible with the flat connections$)$$;$
\item[\rm(3)] for some $k\in\ZZ_{>0}$ rendering $kD$ a $\ZZ$-divisor, $F_m:=h_\ast(mL-kD)$ is a sub-local system of $E_m$ for every $m\in\ZZ_{>0}$.
\end{itemize}
Then $h \colon V\to W$  is a locally constant fibration with respect to $(V,D)$.    
\end{prop}
\begin{proof}
By assumption, the line bundle $L$ induces a relative embedding $i \colon V\hookrightarrow\PP (E_1)$ over $W$ 
such that $i^\ast \OX_{\PP (E_1)}(1)=L$.
Let $p \colon \Unv W\to W$ be the universal cover of $W$. 
Then $p^\ast\!E_1$ is a trivial vector bundle, 
and there are $(r+1)$ global sections $e_0, e_1, \ldots, e_r\in\Coh^0(\Unv W, p^\ast\!E_1)$ that are parallel with respect to (the pullback of) the flat connection $\nabla_{\!E_1}$, where $r+1:=\rank\!E_1$.

Let $\calI_V$  be the ideal sheaf of $V$ in $\PP (E_1)$. 
By the relative Serre vanishing, for a sufficiently large $m$, we have the short exact sequence
\begin{align*}
0\to S_{m}:=g_\ast(\calI_V\otimes\OX_{\PP(E_1)}(m))\to &g_\ast(\OX_{\PP(E_1)}(m)) \\
=&\Sym^m\!E_1\to E_m:=h_\ast(mL)\to 0, 
\end{align*}
where $g$ denotes the natural morphism $\PP (E_1)\to W$. 
Since $\calI_V$ is flat over $W$ (by the flatness of $h$), 
the sheaf $S_m=g_\ast(\calI_V\otimes\OX_{\PP( E_1)}(m))$ is a vector bundle for a sufficiently large $m$  
by \cite[Proposition (3.3), p.\,13, Vol.I\!I]{ACG11}. 
The morphism $g_\ast(\OX_{\PP (E_1)}(m))\to E_m$ is a morphism of local systems by condition (2). 
This indicates that $S_m$ is also a local system and 
the inclusion $S_m\to\Sym^m\!E_1$ is also a morphism of local systems 
since the category of local systems is abelian. 
Hence, we get  global sections $s_1,\ldots,s_{t_m}\in\Coh^0(\Unv W, p^\ast S_m)$ that are parallel with respect to the flat connection $\nabla_{S_m}$, where $t_m:=\rank S_m$. 
Every $s_i \in\Coh^0(\Unv W, p^\ast S_m)$, regarded as a section of $p^\ast\Sym^m\!E_1$, 
is parallel with the pullback of $\nabla_{\Sym^m\!E_1}$.  Hence, we can write
\[
s_i=\sum_{\substack{\alpha=(\alpha_0,\cdots,\alpha_r)\in\ZZ_{\geqslant0}^{r+1} \\ |\alpha|=m}}c_{i,\alpha}\cdot e_0^{\alpha_0}\cdots e_r^{\alpha_r}\,
\]
for some constants $c_{i,\alpha}\in\CC$. 
This implies that the relative embedding $\tilde V:=V\times_W\Unv W$ into $\PP^r\times\Unv W$ over $\Unv W$ is defined by polynomials whose coefficients are independent of $w\in\Unv W$. Hence $\tilde V$ splits into a product $\Unv W\times F$, where $F$ is the general fiber of $h$. 

The vector bundle $E_1$ is induced by a representation $\pi_1(W)\to\GL(r+1)$, 
which gives rise to a representation $\rho \colon \pi_1(W)\to\PGL(r+1)$. 
For $\gamma\in\pi_1(W)$, the automorphism 
$\rho(\gamma) \colon \PP^r\to\PP^r$ sends the fiber $V_w$ at $w\in W$ 
to $V_{\gamma(w)}$ viewed as subvarieties of $\PP^r$. 
As seen before, the defining polynomial of $V_w$ in $\PP^r$ is independent of $w$. 
Hence $\rho(\gamma)$ can be restricted to $F$ so that we get the representation $\rho \colon \pi_1(W)\to\Aut(F)$. 
By construction, the variety $V$ is isomorphic to the quotient of $\tilde V$ 
by the action of $\pi_1(W)$. Hence $h$ is a locally constant fibration.

Moreover, since $F_m=h_\ast(mL-kD)$ is a sub-local system of $E_m$, 
the vector bundle $(h|_{kD})_\ast(mL|_{kD})$ is also a local system.
Then, by the same argument as above, we see that 
the automorphism $\rho(\gamma)\in\Aut(F)$ is also an automorphism of $D_F:=D|_F$ for every $\gamma\in\pi_1(W)$. 
Moreover $D$ is isomorphic to the quotient of $D\times_W\Unv W$ by the action of $\pi_1(W)$. 
Hence $h$ is a locally constant fibration with respect to $(V,D)$.
\end{proof}

\begin{rem}
\label{rmk_prop_flat-lcf}
\begin{itemize}
\item[\rm{(a)}] {\hyperref[prop_flat-lcf]{Proposition \ref*{prop_flat-lcf}}} differs from \cite[Proposition 2.1]{Wang20} at the point that we already assume that the $E_m$ denotes a local system so that we can tackle the non-compact case, and the result of Simpson \cite{Sim92} is not required in the proof; however, we must assume the compatibility condition (2). 
Note that the result is not true without the compatibility condition (2), e.g.,\,if $W$ is a non-compact Riemann surface, then every vector bundle on $W$ is trivial, but there are  non-isotrivial flat families over $W$.  

\item[\rm{(b)}] If we further assume that $W$ is a compact K\"ahler manifold, then the conditions (1), (2), and (3) can
 be replaced with the condition that $E_m$ and $F_m$ are both numerically flat vector bundles 
 by \cite{DPS94,Sim92,Deng17b,Cao13} 
 (cf.\,\cite[Proposition 2.4]{Cao19}, \cite[Proposition 2.8]{CCM19}, and \cite[Proposition 2.1]{Wang20}). 
Indeed, the compatibility conditions are automatically satisfied by \cite[Lemma 4.3.3]{Cao13}).

\item[\rm{(c)}] From the proof, 
we see that it is sufficient to check that the conditions (1), (2), and (3) hold for $m=1$ 
and for some sufficiently large $m=m_0$.
\end{itemize}
\end{rem}

\subsection{Singular Hermitian  metrics on torsion-free sheaves}\label{subsec-sing}

In this subsection, we introduce singular Hermitian  metrics on torsion-free sheaves on normal analytic varieties.

Let $\mathcal{E}$ be a torsion-free sheaf on a normal analytic variety $X$. 
Throughout this paper, the notation $X_{\mathcal{E}}$ denotes 
the maximal locally free locus  of $\mathcal{E}$, 
and $X_{\reg}$ (resp.\,$X_{\sing}$) denotes 
the regular locus (resp.\,the singular locus) of $X$. 
Note that $\codim X_{\sing} \geq 2$ and $\codim (X \setminus X_\mathcal{E}) \geq 2$  
by the normality of $X$ and the torsion-freeness of $\mathcal{E}$.

Let $g$ be a singular Hermitian  metric on $\mathcal{E}$ 
(which means a possibly singular Hermitian  metric  on the vector bundle $\mathcal{E}|_{X_{\reg} \cap X_\mathcal{E} }$).  
See \cite{Raufi15, Pau16, PT18, HPS18} for singular Hermitian  metrics on vector bundles 
and see also \cite[\S 1.4, \S 2.2.4]{Wang-thesis}. 
Let $\theta$ be a smooth $(1,1)$-form on $X$ admitting a local  potential function 
(i.e.,\,$\theta$ can be locally written as $\theta=\ddbar f$ for some smooth function $f$). 
Note that $d$-closed forms do not always admit local potential functions when $X$ has singularities. 
We write 
\begin{align}\label{weak}
\sqrt{-1}\Theta_{g}\succeq \theta \otimes {\rm{id}} \text{ on } X
\end{align}
if the local function $\log |e|_{g^{*}}-f$ is plurisubharmonic (psh) on $X_{\reg} \cap X_\mathcal{E} $ 
for any (holomorphic) local section $e $ of $\mathcal{E}^{*}$, 
where $f$ is a local potential of $\theta$ and 
$g^{*}$ is the induced metric on the dual sheaf $\mathcal{E}^*:=\SheafHom[] (\mathcal{E}, \mathcal{O}_{X})$. 
The notation $\sqrt{-1}\Theta_{g}$, which resembles the curvature, 
does not make sense  since an appropriate definition of curvature is not known in the higher-rank case, 
but  the notation  \eqref{weak} does make sense. 
The function $\log |e|_{g^{*}}-f$  is actually psh on $X$ 
since any psh functions on a Zariski open set $X_{0}$ of $\codim (X \setminus X_{0})\geq 2$ 
can be extended to $X$ by \cite[Satz 3, p.\,181]{GR56}.

\begin{defi}\label{def-weak}
Let $\mathcal{E}$ be a torsion-free sheaf and 
$\omega$ be a positive $(1,1)$-form  on a compact normal analytic variety $X$ admitting a local  potential function.

(1) $\mathcal{E}$
is said to be {\textit{weakly positively curved}} 
if $\mathcal{E}$ admits singular Hermitian  metrics $\{g_\e\}_{\e>0}$ 
such that 
$\sqrt{-1}\Theta_{g_\e} \succeq-\e \omega \otimes \id \text{ on } X$.

(2) $\mathcal{E}$ is said to be \textit{pseudo-effective} if, for any $m \in \mathbb{Z}_{>0}$, there exists 
a singular Hermitian metric $h_m$ on $(\Sym ^m \mathcal{E})^{**}$ such that 
$
\sqrt{-1}\Theta_{h_{m}} \succeq-\omega \otimes \id \text{ on } X. 
$  

\end{defi}

By definition, weakly positively curved sheaves are always pseudo-effective. 
Note that there is an algebraic-geometric characterization of the pseudo-effectivity 
when $\mathcal{E}$ is a locally free sheaf on a smooth projective variety 
(see \cite[Proposition 2.2]{HIM19} and references therein). 
In this paper, we consider the pullback of weakly positively curved sheaves by surjective morphisms. 
If a locally free sheaf is weakly positively curved, its pullback is so (e.g.,\,see \cite{PT18}). 
However, the behavior of torsion-free sheaves under a pullback differs 
from that of locally free sheaves (see  {\hyperref[rem-pull]{Remark \ref*{rem-pull}}} below). 
The following lemma gives a condition that guarantees that torsion-free sheaves satisfy this property.

\begin{lemm}\label{lem-pull}
Let $\phi \colon M \to X$ be a surjective morphism between  
$($not necessarily compact$)$ normal analytic varieties $M$ and $X$. 
Let $\theta$ be a smooth $(1,1)$-form on $X$ admitting a local potential function,  
and let $(\mathcal{E}, g)$ be a torsion-free sheaf on $X$ 
with a singular Hermitian  metric $g$ satisfying $\sqrt{-1}\Theta_g\succeq \theta\otimes\id$ on $X$. 
Assume that the inverse image $\phi^{-1}(X\setminus X_\mathcal{E})$ is of codimension $\geqslant 2$ 
and that the pullback $\phi^{*}(\mathcal{E}^*)$ is a reflexive sheaf on $M$. 

Then, the induced metric $\phi^{*}g$ on  $\phi^{*}\mathcal{E}|_{\phi^{-1}(X_{\mathcal{E}})}$ 
$($which is a vector bundle on $\phi^{-1}(X_{\mathcal{E}})$$)$
can be extended to the singular Hermitian  metric 
on the torsion-free sheaf $\phi^{*}\mathcal{E}/{\rm{Tor}}$ defined by the quotient  by its torsion subsheaf satisfying  that 
$$
\sqrt{-1}\Theta_{\phi^{*}g} \succeq \phi^*\theta \otimes \id \text{ on } M. 
$$ 
\end{lemm}

\begin{rem}\label{rem-pull}
(1) The assumptions in {\hyperref[lem-pull]{Lemma \ref*{lem-pull}}} are automatically  satisfied if 
$\phi \colon M \to X$ is flat   
or if $\mathcal{E}$ is locally free  on $X$. 
Indeed, by $\codim (X\setminus X_\mathcal{E}) \geq 2$, 
the flatness implies $\codim \phi^{-1}(X\setminus X_\mathcal{E}) \geq 2$. 
Furthermore, the dual sheaf of any coherent sheaf is always reflexive, 
and the reflexivity is preserved under pullback by flat morphisms 
(see \cite[Corollary 1.2, Proposition 1.8]{Har80}).

(2) In general, the pullback $\phi^{*}\mathcal{E}$ is not always a torsion-free sheaf. 
Even if we consider $\phi^{*}\mathcal{E}/{\rm{Tor}}$, the conclusion of the lemma does not hold without the assumptions. 
For example, the maximal ideal $\mathfrak{m}_{p} \subset \mathcal{O}_{X,p}$ 
at a given point  $p \in X$ admits a singular Hermitian  metric $g$ 
such that $\sqrt{-1} \Theta_g = 0$ on $X$, 
which is induced by the trivial metric on $\mathcal{O}_{X}$. 
For the blow-up $\phi \colon M:={\rm{Bl}}_{p}(X) \to X$ at the point $p \in X$, 
the quotient sheaf $\phi^{*}\mathfrak{m}_{p}/{\rm{Tor}}$ of the pullback $\phi^{*}\mathfrak{m}_{p}$ is 
the ideal sheaf $\mathcal{O}_{M}(-E)$ associated with the exceptional divisor $E$. 
However, the sheaf $\mathcal{O}_M(-E)$ obviously admits no singular Hermitian  metric with semi-positive curvature. 
\end{rem}

\begin{proof}
Let $e$ be a section of $(\phi^{*}\mathcal{E})^{*}$ on an open set $B$ in $M$. 
Our purpose is to demonstrate that $\log |e| _{\phi^* g^{*}}$ is a $\phi^*\theta$-psh function on $B$ 
(i.e.,\,$\ddbar \log |e| _{\phi^* g^{*}} \geqslant \phi^{*}\theta$). 

The induced metric $\phi^* g^{*}$ is a priori defined on $\phi^{-1}(X_{\reg} \cap X_{\mathcal{E}})$. 
The conclusion of the lemma follows from \cite[Lemma 2.3.2]{PT18}   
when $X$ is smooth and $\mathcal{E}$ is locally free. 
Hence $\log |e| _{\phi^* g^{*}}$ is $\phi^*\theta$-psh on $B \cap \phi^{-1}(X_{\reg} \cap X_{\mathcal{E}})$. 
In general, any $\theta$-psh functions on  a Zariski open set 
can be extended on the ambient space if they are bounded from above.  
Therefore, it suffices to demonstrate that $\log |e| _{\phi^* g^{*}}$ is (locally) bounded from above. 

The assumptions yield the following isomorphisms: 
\begin{align*}
&\Coh^{0}(B, (\phi^{*}\mathcal{E})^{*}) \\
 \simeq &
\Coh^{0}(B \cap \phi^{-1}(X_\mathcal{E}), (\phi^{*}\mathcal{E})^{*}) 
\text{ by $\codim \phi^{-1}(X \setminus X_\mathcal{E}) \geqslant 2$ and reflexivity, }
\\ 
\simeq &
\Coh^{0}(B \cap \phi^{-1}(X_\mathcal{E}), \phi^{*} \mathcal{E}^{*})
\text{ by the local freeness of $ \mathcal{E}$ and $\phi^{*} \mathcal{E}$, }\\
\simeq &
\Coh^{0}(B, \phi^{*} (\mathcal{E}^{*}))
\text{ by $\codim \phi^{-1}(X \setminus X_\mathcal{E}) \geqslant 2$ and  reflexivity,}\\
= &
\ilim[\phi(B) \subset V \text{:open} ] \Coh^{0}(V, \mathcal{E}^{*}) 
\otimes_{\Coh^{0}(V, \mathcal{O}_X)} 
\Coh^{0}(B, \mathcal{O}_M) \text{ by the definition of the functor $\phi^{*}$}. 
\end{align*}
Therefore, the section $e$ can be written as 
$$
e=\sum_{i=1}^{m} s_i \otimes f_i
$$ 
for some $s_i \in \Coh^{0}(V, \mathcal{E}^{*}) $ and $f_i \in \Coh^{0}(B, \mathcal{O}_M)$. 
Then, we have 
$$
|e|^{2}_{{\phi^*g}^*}\leqslant \sum_{i=1}^m |f_i|^{2} | s_i \otimes 1|^{2}_{{\phi^*g}^*}
=\sum_{i=1}^m |f_i|^{2} \phi^{*} |s_i|_{g^*}^{2} \text{ on } B \cap \phi^{-1}(X_{\reg} \cap X_{\mathcal{E}})
$$ 
from the Cauchy-Schwarz inequality and the definition of ${\phi^*g}^*$. 
The function $\log |s_i|_{g^*}$ is $\theta$-psh on $V$ (in particular, bounded from above) 
by $\sqrt{-1}\Theta_g\succeq\theta\otimes \id$. 
Combining this fact with the above inequality, 
we  see that $|e|^{2}_{{\phi^*g}^*}$ is bounded from above on $B$, 
and thus, it can be extended to the $\phi^*\theta$-psh function on $B$. 
\end{proof}

At the end of this subsection, 
we investigate the push-forward  of weakly positively curved sheaves by birational morphisms. 

\begin{lemm}\label{lem-push}
Let $\pi \colon  M \to X$ be a birational morphism 
from a smooth projective variety $M$ to a  $\QQ$-factorial projective variety $X$, 
and let $\mathcal{F}$ be a weakly positively curved sheaf on $M$. 
Then, the push-forward $\pi_{*} \mathcal{F}$ is weakly positively curved  on $X$. 
\end{lemm}
\begin{proof}

By assumption, there exists a singular Hermitian  metric $g_\e$ on $\mathcal{F}$
such that 
$$
\sqrt{-1}\Theta_{g_\e} \succeq-\e \omega_{M} \otimes \id \text{ on } M, 
$$ 
where $\omega_{M}$ is a K\"ahler form on $M$. 
By comparing  $\omega_{M}$ with a K\"ahler form on $X$, 
we investigate the singular Hermitian  metrics $h_{\e}$ on $\pi_{*}\mathcal{F}$ defined by pushing-forward $g_{\e}$.

Let $\omega_{X} $ be a K\"ahler form on $X$ 
defined by $\omega_X:= \Phi ^*\omega_{FS}$, 
where  $\Phi \colon X \to \mathbb{P}(|A|)$ is the embedding associated with a very ample line bundle $A$  on $X$
and $\omega_{FS}$ is the Fubini-Study form on $\mathbb{P}(|A|)$. 
Note that $\omega_X$ locally admits a smooth potential function by construction. 
By $\QQ$-factoriality, we can take a $\pi$-exceptional  effective $\QQ$-divisor $G$ on $M$ 
such that $\pi^* A -G$ is  ample  on $M$ (cf.\,\cite[Lemma 2.9, Complement 2.10, pp.\,73-74]{Kollar07}). 
We may assume that $\omega_M$ is a K\"ahler form representing $c_{1}(\pi^* A -G)$. 
Then, we can take a smooth $(1,1)$-form $\theta$ on $M$ such that 
$$
\theta \in -c_{1}(G) \text{ and }
\omega_M = \pi^{*} \omega_X + \theta. 
$$

We show that there exists a quasi-psh function $\varphi$ on $X$ 
such that 
\begin{align}\label{eq-pot}
\omega_M = \pi^{*} \omega_X  -[G] + \ddbar \pi^{*}\varphi,  
\end{align}
where $[G]$ is the integration current associated with the $\mathbb Q$-divisor $G$. 
This easily follows from the push-forward of currents when $X$ is smooth. 
For the reader's convenience, we present a proof  for  a singular $X$. 
We first take a quasi-psh function $\psi$ on $M$ such that 
$$
\omega_M = \pi^{*} \omega_X  -[G] + \ddbar \psi. 
$$
Let $E$ be the $\pi$-exceptional locus. 
The function $\varphi$ on $X \setminus  \pi(E)$ is defined by $\varphi=\pi_*\psi$ 
via the isomorphism $\pi \colon M\setminus E \simeq X \setminus \pi(E)$. 
Then, a local smooth function $f$ with $\omega_X=\ddbar f$ satisfies that 
$$
\ddbar (\pi^* f +  \pi^*\varphi) = \pi^{*}\omega_{X} + \ddbar \psi = \omega_{M} \text{ locally on }  M \setminus E 
$$ 
by  $\Supp (G) = E$. 
This indicates that $f + \varphi$ is a psh function on $X \setminus \pi(E)$. 
The function $\varphi$ can be extended to the quasi-psh function on $X$ 
since $f$ is smooth and $\pi(E)$ is of codimension $\geqslant 2$. 
Then, the extended function $\varphi$ satisfies that $\psi=\pi^* \varphi$ on $M$ 
since $\psi$ and $\pi^* \varphi$  are quasi-psh, which leads to the desired equality \eqref{eq-pot}.

We define the singular Hermitian  metric $h_\e$ 
on $\pi_*\mathcal{F}|_{X \setminus \pi(E)}$
by $h_{\e}:=\pi_{*}g_{\e}$. 
We will  construct weakly positively curved metrics on $\pi_*\mathcal{F}$
by modifying $h_{\e}$ with $\varphi$. 
For simplicity of the notation, 
we used the same notation $X$ (resp.\,$M$) to denote a sufficiently small open set in $X$ 
(resp.\,its inverse image by $\pi$). 
Let us consider the function 
$$\text{
$\varphi_\e:=\log |e|_{h_{\e}^*}$ on $X \setminus \pi(E)$ 
}
$$
for a local section $e$ of $(\pi_*\mathcal{F})^{*}$. 
We now have 
\begin{align*}
&\Coh^{0}(X, (\pi_*\mathcal{F})^*)\\
\simeq & \Coh^{0}(X \setminus \pi(E), (\pi_*\mathcal{F})^*) 
\text{ by } \codim \pi(E)\geqslant 2
\text{ and reflexivity,}\\
\simeq & \Coh^{0}(M\setminus E, \mathcal{F}^*)
\text{ by the isomorphism } \pi \colon M\setminus E \simeq X \setminus \pi(E). 
\end{align*}
This  indicates that the section $e$ of $(\pi_\ast\mathcal{F})^{*}$ can be identified with the section in 
$\Coh^{0}(M\setminus E, \mathcal{F}^*)$, which we denote by $\pi^{*}e$. 
Then, by  definition, we have 
$$
\pi^{*}\varphi_\e=\pi^* (\log |e|_{h_\e}^* )=\log |\pi^*e|_{\pi^{*}h_\e^{*}}=\log|\pi^{*}e|_{g_\e} \text{ on } 
M \setminus E. 
$$
Hence, we obtain 
$$
\ddbar \pi^{*}\varphi_{\e}=
\ddbar \log|\pi^{*}e|_{g_\e} \geqslant -\e \omega_M =-\e (\pi^*\omega_X + \ddbar \pi^*\varphi )
\text{ on } M \setminus E
$$
by the selection of $g_{\e}$ and equality \eqref{eq-pot}. 
This equality indicates that 
$$
\ddbar (\varphi_{\e} + \e \varphi)  \geqslant -\e \omega_X \text{ on } X\setminus \pi(E). 
$$
Hence $\varphi_{\e} + \e \varphi$ can be extended to 
the $\e \omega_X$-psh function on $X$
by $\codim \pi(E) \geqslant 2$. 
Note that $\varphi_{\e} + \e \varphi$ can be extended
because of $\varphi$ (which may have the pole), 
but this extension is not expected for $\varphi_{\e}$ itself. 
Thus, the metric $H_\e:=h_\e e^{-\e \varphi}$ on $\pi_* \mathcal{F}$ 
satisfies the definition of weakly positively curved metrics. 
 \end{proof}

\subsection{Finite quasi-\'etale covers}
\label{subsec_q-et}

In this subsection, following \cite{GKP16b}, 
we summarize the definition and basic properties of quasi-\'etale  covers.

\begin{defi}\label{def-finite}
Let $\nu \colon X \to Y$ be a surjective morphism between normal analytic varieties $X$ and $Y$. 

\begin{itemize}

\item[(1)]
$\nu \colon X \to Y$  is called a \textit{cover} 
(resp.\,\textit{finite cover}) of $Y$ 
if every fiber of  $\nu \colon X \to Y$  is the set of discrete points 
(resp.\,$\nu \colon X \to Y$ is a finite morphism).

\item[(2)]
$\nu \colon X \to Y$ is said to be \textit{quasi-\'etale}  if 
there exists a Zariski closed set $Z\subset X$ of codimension $\geqslant 2$ 
such that the induced morphism $\nu|_{X\setminus Z} \colon X\setminus Z \to Y$ is \'etale. 

\item[(3)] $\nu \colon X \to Y$ is said to be \textit{maximally quasi-\'etale}  if 
it is a finite quasi-\'etale cover  such that 
any finite \'etale cover of $X_{\reg}$ extends to a finite \'etale cover of $X$,
i.e.,\,the natural morphism 
$$
i_* \colon \hat{\pi}_1(X_{\reg}) \to \hat{\pi}_1(X)
$$
between the \'etale fundamental groups (which are the profinite completion of the topological fundamental groups) is isomorphic. 
We simply say that $X$ is maximally quasi-\'etale when $X$ itself satisfies the above condition.

\end{itemize}
\end{defi}

Note that the category of finite quasi-\'etale covers of $X$ 
is equivalent to that of finite index subgroups in the \'etale fundamental group $\hat{\pi}_1(X_{\reg})$  of $X_{\reg}$. 
For potentially klt projective varieties, 
a maximally quasi-\'etale cover always exists by \cite[Theorem 1.5]{GKP16a} 
and this property is preserved under birational morphisms (see {\hyperref[thm-mqe]{Theorem \ref{thm-mqe}}} (2)), 
which plays an important role in {\hyperref[sec-flat]{Section \ref*{sec-flat}}}.

\begin{theo} \label{thm-mqe}

Let $X$ be a potentially klt projective variety. Then, we have$:$
\begin{itemize}
\item[(1)]There exists a maximally quasi-\'etale cover Galois cover $\nu\colon \bar X \to X$. 

\item[(2)]
Let $X'$ be a potentially klt projective variety and $\pi\colon X' \to X$ be a birational morphism. 
If $X$ is a maximally quasi-\'etale cover, then so is $X'$. 
\end{itemize}
\end{theo}
\begin{proof}
Conclusion (1) is a direct consequence of \cite[Theorem 1.5]{GKP16a}. 
We check conclusion (2). 
Take a Zariski open set $U \subset X_{\reg}$ such that 
$\codim (X \setminus U) \geq 2$ holds and $\pi \colon X' \to X$ is isomorphic over $U$. 
Set $V:=\pi^{-1}(U) \subset X'$ and consider the following commutative diagram: 
\[
\xymatrix{
& \pi_1(V)  \ar@{->>}[r]^{j_* \quad }  \ar[d]_{\pi_*}^{\cong} & \pi_1(X'_{\reg}) 
\ar@{->>}[r] & \pi_1(X') \ar[d]_{\pi_*}^{\cong}  \\
& \pi_1(U) \ar[r]^{i_* \quad }_{\cong \quad } & \pi_1(X_{\reg}) \ar@{->>}[r]  & \pi_1(X). 
}
\]
The vertical arrows  are isomorphic by \cite{Tak03}. 
All the horizontal arrows are surjective since $X$ and $X'$ are normal, 
and further, the morphism $i_*\colon  \pi_1(U)  \to \pi_1(X_{\reg})$ induced by the natural inclusion 
$i\colon U \hookrightarrow X_{\reg}$ 
is isomorphic by $\codim (X_{\reg} \setminus U) \geq 2$. 
Then, by taking the profinite completion (which is a right-exact functor), 
we obtain the following diagram:
\[
\xymatrix{
& \hat \pi_1(V)  \ar@{->>}[r]^{j_*\quad}  \ar[d]_{\pi_*}^{\cong} & \hat \pi_1(X'_{\reg}) \ar@{->>}[r] & \hat \pi_1(X') \ar[d]_{\pi_*}^{\cong}  \\
& \hat \pi_1(U) \ar[r]^{i_*\quad}_{\cong\quad} & \hat \pi_1(X_{\reg}) \ar[r]_{\cong}  & \hat \pi_1(X). 
}
\]
Note that  $\hat \pi_1(X_{\reg}) \to  \hat \pi_1(X)$ is isomorphic 
since $X$ is a maximally quasi-\'etale cover. 
This implies that $\hat \pi_1(X'_{\reg}) \to  \hat \pi_1(X')$ is also isomorphic. 
\end{proof}

The following lemma is an elementary result obtained from the Stein factorization, 
and we thus omit the proof. 

\begin{lemm}
\label{lemma_term-model}
Let $g\colon X' \to X$ be a birational morphism and $\eta'\colon X_{1}' \to X'$ be a finite quasi-\'etale  cover 
between normal projective varieties. 
Then, there exists a projective variety $X_{1}$ 
with a finite quasi-\'etale cover  $\eta_{1}\colon X_{1} \to X$ 
and a birational morphism $g_{1}\colon X_{1}' \to X_{1}$ 
such that $\eta_{1}\circ g_{1}=g \circ \eta'$. 
\end{lemm}

\section{Flat connections on direct image sheaves}\label{sec-flat}

In this section, we study projective klt pairs with nef anti-log canonical divisor 
and  direct image sheaves  appropriately defined by their MRC fibrations. 
This section is devoted to the proof of {\hyperref[thm-flat]{Theorem \ref*{thm-flat}}}, 
which states  that the direct image sheaves satisfy a certain flatness 
up to $\QQ$-factorializations and finite quasi-\'etale covers. 

For a given klt pair $(X, \Delta)$, 
we take a maximally quasi-\'etale cover $\nu\colon \bar X \to X$ by {\hyperref[thm-mqe]{Theorem \ref{thm-mqe}}} (1) 
and a $\QQ$-factorial log terminal model $\pi\colon (\bar X^{\qf}, \bar \Delta\qf) \to (\bar X, \bar\Delta)$ by \cite[Corollary 1.4.3]{BCHM10}, 
where $\bar \Delta\qf$ and $\bar \Delta$ are $\mathbb{Q}$-divisors defined by the pullbacks.
Then, the pair $(\bar X\qf, \bar \Delta\qf)$ is a $\QQ$-factorial terminal pair 
and $\bar X\qf$ is a maximally quasi-\'etale cover by {\hyperref[thm-mqe]{Theorem \ref*{thm-mqe}} (2).  
We will apply {\hyperref[thm-flat]{Theorem \ref*{thm-flat}}}  to $(\bar X^{\qf}, \bar \Delta^{\qf})$
to derive the holomorphicity and local constancy of MRC fibrations of $\bar X^{\qf}$, 
and prove that the MRC fibration of $\bar X^{\qf}$ induces the desired MRC fibration for $(\bar X, \bar \Delta)$, 
which is discussed in {\hyperref[sec-MRC]{Section \ref*{sec-MRC}}}.

The proof of {\hyperref[thm-flat]{Theorem \ref*{thm-flat}}} clarifies 
why we need a finite quasi-\'etale cover of $X$ in {\hyperref[thm-main]{Theorem \ref*{thm-main}}}, 
whereas the quasi-\'etale cover  never appears when $X_{\reg}$ is  simply connected (see \cite{Wang20}) 
or when $X$ is smooth (see \cite{CCM19}). 
{\hyperref[thm-flat]{Theorem \ref*{thm-flat}}} requires a deeper insight  into nef anti-log canonical divisors  
and a more involved argument 
than what has been presented in the previous works \cite{Cao19, CH19, CCM19, Wang20}.

\subsection{Setting and goal for this section}
\label{subsec-setting}

In this subsection, we give a precise formulation of {\hyperref[thm-flat]{Theorem \ref*{thm-flat}}} 
after we describe our setting in this section.

\begin{setting}\label{setting}
Let $(X, \Delta)$ be a projective $\QQ$-factorial terminal  pair 
such that $X$ is a maximally quasi-\'etale cover and  
the anti-log canonical divisor $-(K_X+\Delta)$ is nef. 
 Let $\psi \colon X \dashrightarrow Y$ be an MRC fibration of $X$ 
to a smooth projective variety $Y$. 
Let $\pi \colon M \to X$ be a resolution of singularities of $X$ and indeterminacies of $\psi \colon X \dashrightarrow Y$ 
with a morphism $\phi \colon M \to Y$ in the following diagram:  
\begin{align}\label{comm-start}
\xymatrix{
M  \ar[rr]^{\pi} \ar[rd]_{\phi}&   & X \ar@{.>}[ld]^{\psi} \\
  & Y. & 
}
\end{align}

\noindent Let $A$ be a (sufficiently) ample line bundle on $X$ and set $E:=\Exc(\pi)$. 
(In fact, {\hyperref[thm-flat]{Theorem \ref*{thm-flat}}} is proved 
for any effective divisor $E$ whose support coincides with $\Exc(\pi)$, 
but in this paper we set $E:=\Exc(\pi)$ for simplicity.)
For a fixed  integer $c \in \mathbb{Z}_{>0}$ (which we take to be large enough later), 
we consider the determinant sheaf 
$$\det \phi_* \mathcal{O}_{M}(\pi^*A + cE):= 
\Big (\bigwedge^{r} \big( \phi_* \mathcal{O}_{M}(\pi^*A + cE) \big) \Big )^{**}, 
$$  
where $r$ is the rank of the direct image sheaf $\phi_* \mathcal{O}_{M}(\pi^*A + cE)$. 
Note that  this determinant sheaf  is an invertible sheaf (a line bundle)
since it is reflexive and $Y$ is smooth. 
We define the $\phi$-big line bundle $L_{m}$ on $M$  by   
$$
L_{m}:=\mathcal{O}_{M}(m(\pi^*A + cE)) 
-\frac{m}{r}\phi^* \det \phi_* \mathcal{O}_{M}(\pi^*A + cE), 
$$
where $m$ is a positive integer with $m/r \in \mathbb{Z}$. 
Here, we used the additive notation for tensor products and 
we can regard $L_{m}$ as a line bundle by $m/r \in \mathbb{Z}$ 
(see {\hyperref[subsec-notation]{Subsection \ref*{subsec-notation}}}). 
Furthermore, we define the direct image sheaf  $\mathcal{V}_{m}$ on $Y$  by 
$$
\mathcal{V}_{m}:=\phi_* \mathcal{O}_{M}( L_{m} ).
$$
The subscript $m$ in $\mathcal{V}_m$ and $L_{m}$ is not important in most of this section, 
and we thus often omit the subscript $m$ to simplify the notation. 
Let $Y_0 \subset Y$ be the maximal Zariski open set in $Y$ 
satisfying the following properties: 
\begin{itemize}
\item $\phi \colon M\to Y$ is a flat  morphism over $Y_0$. 
\item $\phi^{*}P$ is not $\pi$-exceptional for any prime divisor $P$ on $Y_0$. 
\end{itemize}
\end{setting}

This section aims to prove that 
$\mathcal{V}_{m}$ satisfies a certain flatness on $Y_{0}$.

\begin{theo}\label{thm-flat}
Consider the same situation as in {\hyperref[setting]{Setting \ref*{setting}}}. 
Then, for some fixed $c \in \mathbb{Z}_{>0}$ and 
for every $m \in \mathbb{Z}_{>0}$ with $m/r \in \mathbb{Z}$, 
there exists a Zariski closed set  $C_{m} \subset Y$ such that  
\begin{itemize}
\item[$\bullet$] $C_{m} \subset Y$  is of codimension $\geqslant2$$;$ 
\item[$\bullet$] $\mathcal{V}_{m}$ is locally free on $Y_{0} \setminus C_{m}$$;$ 
\item[$\bullet$] $\mathcal{V}_{m}$ admits a flat connection on $Y_{0} \setminus C_{m}$. 
\end{itemize}

\end{theo}

The proof of {\hyperref[thm-flat]{Theorem \ref*{thm-flat}}} is divided into three subsections. 
Throughout this section, we keep {\hyperref[setting]{Setting \ref*{setting}}} and promise 
that $m$ always satisfies $m/r \in \mathbb{Z}$.

\subsection{Birational semi-stability for MRC fibrations}
\label{subsec-semistability}

In this subsection, we confirm 
a certain birational semi-stability for the MRC fibration $\psi \colon X \dashrightarrow Y$. 
Such a semi-stability result, which essentially follows from \cite[Main Theorem]{Zhang05}, 
has been explicitly formulated in \cite{CH19}, \cite[Theorem 3.2]{CCM19}, and \cite[Proposition 3.1]{Wang20}. 
{\hyperref[prop-2.2]{Proposition \ref*{prop-2.2}}} is a slight refinement of the above results.

\begin{prop}\label{prop-2.2}
The following statements hold$:$ 
\begin{itemize}
\item[\rm(a)]$ \phi^* N_Y$ is $\pi$-exceptional for any effective divisor $N_Y$ 
on $Y$ with $N_Y \sim_{\QQ} K_Y$. 
\item[\rm(b)] The Kodaira dimension of $Y$ is zero $($i.e.,\,$\kappa(K_Y)=0$$)$. 
\item[\rm(c)] $\pi(\phi^{-1}(Y \setminus Y_0))$ is of codimension $\geqslant 2$. 
\item[\rm(d)] $Y_0$ has the generalized Liouville property in the following sense$:$ 
For a flat vector bundle  $(\calH_0,\nabla_0)$ on $Y_0$  with the following condition $\rm(\bullet)$, 
every global section of $\calH_0$ is parallel with respect to $\nabla_0$. 
\begin{itemize}
\item[\rm($\bullet$)] 
There exists a numerically flat vector bundle $\calH$ on $X$ such that 
$$
\text{$(\phi^\ast\calH_0, \phi^\ast \nabla_0) \simeq (\pi^\ast\calH, \nabla)$ on $M_{0}:=\phi\inv(Y_0)$, 
}
$$
where $\phi^\ast \nabla_0$  is the connection on $\phi^\ast\calH_0$ 
defined by the pullback and $\nabla$ is the $($unique$)$ flat connection 
on $\pi^\ast\calH$ defined by \cite[Corollary 3.10 and the discussion thereafter]{Sim92} 
$($which is compatible with the filtration given by \cite[Theorem 1.18]{DPS94}$)$. 
\end{itemize}

\item[\rm(e)] $\psi$ is semi-stable in codimension one in the following sense$:$ 
Let $P$ be a prime divisor  on $Y_0$ and $\phi^*P=\sum_i c_i P_i$ 
be the irreducible decomposition of $\phi^*P$. 
Then, any non-reduced component $P_i$ {\rm (}i.e.,\,a component $P_{i}$ with $c_i > 1${\rm )} is $\pi$-exceptional.
\item[\rm(f)] $\Delta$ is horizontal with respect to $\psi$ 
$($i.e.,\,$Y=\phi( \pi^{-1}_{*} \Delta_{i})$ for any component $\Delta_{i}$ of $\Delta$$)$.

\end{itemize} 
\end{prop}

\begin{rem}\label{rem-qabelian}
In the proof, we assume that $\psi \colon X \dashrightarrow Y$ is an MRC fibration 
only to deduce that $Y$ is not uniruled. 
In fact,  we can obtain all the conclusions 
for an almost holomorphic map $\psi \colon X \dashrightarrow Y$ with a non-uniruled $Y$ 
under the weaker assumption that $(X, \Delta)$ is log canonical. 
The case of $\calH=\OX_{Y_0}$ in property (d) 
is nothing but \cite[Proposition 3.1(d)]{Wang20}.
\end{rem}

\begin{proof}
Note that the base variety $Y$ of MRC fibrations is not uniruled by \cite{GHS03}. 
Hence, all the conclusions except for (d) follow from \cite[Proposition 3.1]{Wang20}. 

The generalized Liouville property (i.e.\,conclusion (d)) essentially follows  from property (c). 
Let $s\in\Coh^0(Y_0,\calH_0)$. Then $s$ is parallel with respect to $\nabla_0$ if and only if $\phi^\ast\!s$ is parallel with respect to $\phi^\ast\nabla_0=\nabla|_{\phi\inv(Y_0)}$. 
Since the complement of $\pi(\phi\inv(Y_0)\setminus  E)$ in $X$ is of codimension $2$ by property (c), 
$$\text{
the section 
$\phi^\ast\!s|_{\phi\inv(Y_0)\setminus E}$ on $\phi \inv(Y_0) \setminus  E \simeq  \pi (\phi^{-1} (Y_{0}) \setminus  E)$ 
}
$$
induces a section $\sigma\in\Coh^0(X,\calH)$ by reflexivity (via the isomorphism $\pi|_{\phi\inv(Y_0)\setminus  E}$). Then $\pi^\ast\!\sigma$ is parallel with respect to $\nabla$ by \cite[Theorem 4.3.3]{Cao13}, and hence so is $\phi^\ast\!s$. 
\end{proof}

From the next subsection, we examine the local freeness and 
numerical flatness of the reflexive sheaf $\mathcal{E}_{m}$ on $X$ defined by   
$$
\mathcal{E}_{m}:= (\pi_* \phi^* \mathcal{V}_{m})^{* *}. 
$$ 
At the end of this subsection, following \cite{DPS94}, 
we recall the definition of numerically flat vector bundles. 
Note that we do not define the numerical flatness for non-locally free sheaves. 

\begin{defi}[{Numerically flat locally free sheaves}]\label{def-numflat}
Let $\mathcal E$ be a locally free sheaf on  a projective variety. 
The  sheaf $\mathcal E$  is said to be \textit{numerically effective $($nef$)$} 
if the hyperplane bundle $\mathcal{O}_{\mathbb{P}(\mathcal E)}(1)$  on the projective space bundle $\mathbb{P}(\mathcal E)$ 
is a nef line bundle. 
Furthermore, the sheaf $\mathcal E$  is said to be \textit{numerically flat} if   
$\mathcal E$  is nef and $c_{1}(\mathcal E)=0$. 
\end{defi}

\subsection{Positivity of singular Hermitian metrics on direct image sheaves}
\label{subsec-posi}

The purpose of this subsection is to prove {\hyperref[cor-free]{Proposition \ref*{cor-free}}}, 
which states that $\mathcal{E}_{m}$ is a numerically flat vector bundle on $X$. 
In the proof of {\hyperref[cor-free]{Proposition \ref*{cor-free}}}, 
we use the assumption that $X$ is a maximally quasi-\'etale cover 
and smooth in codimension two.

We first prove {\hyperref[thm-free]{Proposition \ref*{thm-free}}} 
by applying {\hyperref[lem-push]{Lemma \ref*{lem-push}}} and the positivity of direct image sheaves. 
In the proof of {\hyperref[thm-free]{Proposition \ref*{thm-free}}}, 
we use the assumption that $X$ is  $\QQ$-factorial.

\begin{prop}\label{thm-free}
There exists a sufficiently large integer $c\in\mathbb{Z}_{>0}$ 
$($recalling that $c$ is an integer appearing in the definition of $\mathcal{V}_{m}$$)$
such that every $m \in \mathbb{Z}_{>0}$ the following properties holds$:$ 
\begin{itemize}
\item[$(1)$] $\mathcal{E}_{m}$ satisfies that $c_1(\mathcal{E}_{m})=0$. 
\item[$(2)$] $\mathcal{E}_{m}$ is weakly positively curved on $X$. 
\end{itemize}
\end{prop}

\begin{proof}[Proof of {\hyperref[thm-free]{Proposition \ref*{thm-free}}}]
We deduce (1) and (2) from \cite[Proposition 3.5, Lemma 3.1]{Wang20} and 
Lemmas \ref{lem-pull} and \ref{lem-push}. 
Note that the assumption of \cite[Proposition 3.5, Lemma 3.1]{Wang20} 
is satisfied thanks to the $\QQ$-factoriality of $X$. 

We first prove conclusion (1). 
By \cite[Proposition 3.5]{Wang20} (or \cite[Proposition 3.9]{CCM19}), 
we have already known that $\mathcal{V}_{m}$ satisfies that $c_1(\pi_*\phi^* \mathcal{V}_{m})=0$. 
Strictly speaking, for some sufficiently large integer $c\in\mathbb{Z}_{>0}$, 
we obtain $c_1(\pi_*\phi^* \mathcal{V}_{m})=0$. 
 The sheaf $\phi^* {\mathcal{V}_{m}}$ is reflexive on $M_{0}=\phi^{-1}(Y_{0})$
since $\phi \colon M \to Y$ is flat over $Y_{0}$ by \cite[Proposition 1.8]{Har80}. 
Furthermore, the Zariski closed subset $\pi( M \setminus M_{0})$ is of codimension $\geqslant 2$ 
by property (c) of {\hyperref[prop-2.2]{Proposition \ref*{prop-2.2}}}. 
Hence, we obtain 
$$
\pi_* {\phi}^* {\mathcal{V}_{m}}=(\pi_* {\phi}^* {\mathcal{V}_{m}})^{**}=\mathcal{E}_{m}  \text{ on } 
X \setminus \pi(\Exc (\pi))
$$ 
from the isomorphism $\pi  \colon M \setminus \Exc (\pi) \simeq X \setminus \pi (\Exc (\pi))$. 
Then, we see that 
$
0= 
c_{1}(\pi_* {\phi}^* {\mathcal{V}_{m}})=c_{1}(\mathcal{E}_{m}) 
$
by $\codim \pi (\Exc(\pi)) \geqslant 2$. 

We finally prove conclusion (2).
By \cite[Lemma 3.1]{Wang20},  
the sheaf $\mathcal{V}_{m}$ is weakly positively curved on $Y_0$ in the following sense: 
For any $\e>0$, there exists a singular Hermitian  metric $b_\e$ on $\mathcal{V}_{m}|_{Y_{0}}$
such that 
$$
\sqrt{-1}\Theta_{b_\e} \succeq-\e \omega_{Y} \otimes \id \text{ on } Y_{0} 
$$  
for some K\"ahler form $\omega_{Y} $ on $Y$. 
The key point here is that $\omega_{Y} $ in the above inequality 
is a K\"ahler form on $Y$ (not only on $Y_{0}$).   
Then, we  deduce that $\phi^* \mathcal{V}_{m}$ is weakly positively curved on $M\setminus E$ in the sense that  
\begin{align}\label{ineq3}
\sqrt{-1} \Theta_{g_{\e}} \succeq  -\e \omega_{M} \otimes \id \text{ on } M\setminus E, 
\end{align}
for some singular Hermitian  metric $g_\e$  on $\phi^* \mathcal{V}_{m}|_{M\setminus E}$, 
where $\omega_M$ is a fixed K\"ahler form on $M$. 
Indeed, the induced morphism $\phi\colon M_{0}=\phi^{-1}(Y_{0}) \to Y_{0}$  
satisfies the assumptions of {\hyperref[lem-pull]{Lemma \ref*{lem-pull}}} (see also {\hyperref[rem-pull]{Remark \ref*{rem-pull}}}). 
By applying {\hyperref[lem-pull]{Lemma \ref*{lem-pull}}} to $ \mathcal{V}_{m}$ 
(which is weakly positively curved  on $Y_{0}$), 
we see that \eqref{ineq3} holds on $M_0$. 
Here, we use the fact that $\phi^{*} \omega_{Y} \leqslant  k \omega_{M}$ holds on $M$ 
for a sufficiently large $k$. 
Any $\e \omega_M$-psh functions  on the complement of a Zariski closed set  
of codimension $\geqslant2$  can  be automatically extended to the ambient variety by \cite[Satz 3, p.\,181]{GR56}. 
By applying this fact to the non-divisorial part of $M \setminus M_0$, 
we see that \eqref{ineq3} holds  outside the divisorial components of $M \setminus M_0$. 
The divisorial components of $M \setminus M_0$ are contained in $\Exc(\pi)$ 
by property (c) of {\hyperref[prop-2.2]{Proposition \ref*{prop-2.2}}}.  
We therefore conclude that $\pi_* {\phi}^* {\mathcal{V}_{m}}$ (and thus $\mathcal{E}_{m}$)   
is weakly positively curved on $X \setminus \pi(\Exc(\pi))$. 
Then conclusion (2) follows from $\codim \pi(\Exc(\pi))\geqslant 2$. 
\end{proof}

Hereafter, we fix $c\in\mathbb{Z}_{>0}$  satisfying {\hyperref[thm-free]{Proposition \ref*{thm-free}}}. 
We now discuss the numerical flatness of pseudo-effective sheaves with vanishing first Chern class. 
If $X$ is smooth, then a pseudo-effective reflexive sheaf with vanishing first Chern class is a fortiori locally free and numerically flat by \cite{HIM19,Wu20}. 
This is nevertheless not true  if $X$ has singularities as in the following example: 

\begin{ex}\label{ex-qabelian}
Let $X$ be a rationally connected $Q$-abelian variety, 
i.e.,\,there exists a finite quasi-\'etale cover $\nu \colon A \to X$ by an abelian variety $A$  
(see \cite{Cam10}, \cite[\S 2, Corollary 24, p.\,193]{KL09} for details of $Q$-abelian varieties). 
Then, the tangent sheaf $T_{X}$ admits a flat Hermitian metric on $X_{\reg}$ (in particular, $T_X$ is pseudo-effective) 
since $T_{X_{\reg}}$ is an \'etale quotient of the trivial tangent bundle on the abelian variety. 
Furthermore, the variety $X$ has klt singularities by \cite[Proposition 5.20, p.160]{KM98}  and 
the canonical divisor $K_X$ is $\QQ$-linearly trivial by $p^* K_X \sim_{\QQ} K_A=\mathcal{O}_{A}$. 
Nevertheless, the tangent sheaf $T_{X}$ never is locally free. 
Indeed, if $T_{X}$  is locally free, the variety $X$ is smooth, which contradicts the rational connectedness. 

This example also shows that the fiber dimension of the MRC fibration is not preserved under taking quasi-\'etale covers. 
Indeed, the constant map $X \to \{*\}$ is an MRC fibration of $X$, 
but the identity map $\text{id} \colon A \to A$ is an MRC fibration of the quasi-\'etale cover $A$. 
\end{ex}

Taking a finite quasi-\'etale cover, we can obtain the local freeness and numerical flatness as in \cite{HIM19,Wu20}.

\begin{theo}[{\cite[Theorem 1.8]{HP19}, \cite[Theorem 1.20]{GKP16b}}]
\label{thm-etaleflat}
Let $\mathcal{F}$ be a reflexive sheaf on a normal projective variety $Z$. 
Assume that $Z$ is potentially klt and smooth in codimension two, and 
that $\mathcal{F}$ is pseudo-effective $($which is satisfied if it is weakly positively curved$)$ 
and satisfies $c_{1}(\mathcal{F}) =0$. 
Then, there exists a finite quasi-\'etale Galois cover $\nu\colon \bar Z \to Z$ with normal $\bar Z$  
such that $(\nu^* \mathcal{F})^{* *}$ is  locally free and numerically flat on $\bar Z$. 
\end{theo}
\begin{proof}
The same conclusion has been proved in \cite[Theorem 1.8]{HP19} when $\mathcal{F}$ is almost nef. 
We can easily see that the same proof works when $F$ is pseudo-effective. 
(Strictly speaking, the cited result \cite[Theorem 1.8]{HP19} assumes that $X$ is klt, 
but this assumption can be relaxed to potentially klt (see \cite[Theorem 1.20]{GKP16b}). 
\end{proof}

\begin{prop}\label{cor-free}
The sheaf $\mathcal{E}_{m}$ is  locally free  and numerically flat on $X$ for every $m \in \mathbb{Z}_{>0}$.  
\end{prop}

\begin{proof}
When $X$ is smooth, 
the sheaf $\mathcal{E}_{m}$ is locally free and numerically flat on $X$ by \cite[Theorem 1.2]{HIM19}. 
When $X$ is singular, we need to take a quasi-\'etale cover 
to obtain the same conclusion for the pullback of $\mathcal{E}_{m}$, 
as explained in {\hyperref[ex-qabelian]{Example \ref{ex-qabelian}}} and {\hyperref[thm-etaleflat]{Theorem \ref{thm-etaleflat}}}. 
Nevertheless, since we are assuming that $X$ is a maximally quasi-\'etale cover, 
we obtain the desired conclusion without taking any quasi-\'etale cover.

We first show that it suffices  to prove that 
$\mathcal{E}_{m}$ is  locally free  on $X$ for every $m \in \mathbb{Z}_{>0}$. 
For the resolution of singularities $\pi \colon M \to  X$, 
the pullback $\pi^{*}  \mathcal{E}_{m}$ also satisfies the conclusions of \hyperref[thm-free]{Proposition \ref*{thm-free}}. 
Here, we essentially used the local freeness. 
Hence $\pi^{*} \mathcal{E}_{m}$ is numerically flat 
by \cite[Theorem 1.2]{HIM19} (which corresponds the smooth case of  {\hyperref[thm-etaleflat]{Theorem \ref{thm-etaleflat}}})
and so is $ \mathcal{E}_{m}$.

Since $(X, \Delta)$ is a $\QQ$-factorial terminal pair, 
the variety $X$ has terminal singularities. 
Hence $X$ is smooth in codimension two. 
This indicates that $X$ and $\mathcal{E}_{m}$ satisfy the assumptions of {\hyperref[thm-etaleflat]{Theorem \ref*{thm-etaleflat}}}. 
By {\hyperref[thm-etaleflat]{Theorem \ref*{thm-etaleflat}}},  
for each $m \in \mathbb{Z}_{>0}$, 
we can take a finite quasi-\'etale Galois cover 
$\nu_{m}\colon X_{m} \to X$ such that $(\nu_{m}^* \mathcal{E}_{m})^{* *}$ is locally free. 
Since $X$ is a maximally quasi-\'etale cover, the cover $\nu_{m}\colon X_{m} \to X$ is actually \'etale. 
This implies that  $\mathcal{E}_{m} $ is locally free on $X$. 
\end{proof}

\begin{rem}\label{rem-sub}
Hereafter, the dependence of $\mathcal{E}_{m}$ and $\mathcal{V}_{m}$ on $m$ is negligible, and we thus omit the subscript $m$. 
 \end{rem}

\subsection{Rationally connected fibers and flat connections}
\label{subsec-flat}

In this subsection, we prove {\hyperref[thm-flat]{Theorem \ref*{thm-flat}}} by applying {\hyperref[cor-free]{Proposition \ref*{cor-free}}}.

\begin{proof}[Proof of {\hyperref[thm-flat]{Theorem \ref*{thm-flat}}}]
We omit the subscript $m$ in $\mathcal V_{m}$ and $\mathcal E_{m}$ in the proof. 
The variety $X$ in our setting may have singularities. Thus,  
the $\pi$-exceptional locus $\Exc (\pi)$ may be dominant over $Y$, 
which causes a problem in proving the theorem. 
To overcome this difficulty, we consider 
the normalization $\Gamma$ of the graph of $\psi \colon X \dashrightarrow Y$. 
Both $\phi \colon M \to Y$ and $\pi \colon M \to X$ factorize through $\Gamma$, 
equipped with the morphisms $\varphi \colon \Gamma \to Y$, $\mu \colon \Gamma \to X$, and $\gamma \colon M \to \Gamma$ 
in the following diagram: 

\begin{align}\label{comm0}\xymatrix{
M \ar[dr]_{\gamma} \ar@/^20pt/[drrr]^\pi \ar@/_20pt/[ddr]_{ \phi}& &\\
&\Gamma \ar[d]_{\varphi} \ar[rr]^{\mu} && X \ar@{.>}[lld]^{\psi} \\
&Y.  && 
}
\end{align} 
The $\pi$-exceptional locus $\Exc (\pi)$ may be dominant over $Y$. 
However, the $\mu$-exceptional locus $E_\Gamma:=\Exc (\mu)$ 
is not dominant over $Y$ since $\psi \colon X \dashrightarrow Y$ is almost holomorphic. 

Let $B \subset Y$ be the union of the non-locally free locus of $\mathcal{V}$ 
and non-flat locus of $\varphi \colon \Gamma \to Y$. 
Note that $B$ is of  codimension $\geqslant 2$ by the torsion-freeness of $\calV$ and the normality of $Y$. 
We first confirm the following claim:

\begin{claim}\label{claim2}
We have 
$$
\pi^* \mathcal{E} = \phi^{*}\mathcal{V} 
\text{ on } M \setminus  (\gamma^{-1}(E_\Gamma) \cup  \phi^{-1}(B)). 
$$
\end{claim}
\begin{proof}
The sheaf $\varphi^* \mathcal{V}$ is locally free on $\Gamma \setminus \varphi^{-1}(B)$. 
Hence, we  obtain 
$$
\gamma_* \gamma^*  \varphi^* \mathcal{V}=\varphi^* \mathcal{V} 
\text{ on } \Gamma \setminus \varphi^{-1}(B)
$$ 
by applying the projection formula 
to the algebraic fiber space 
$$\gamma \colon M \setminus \phi^{-1}(B) \to \Gamma \setminus \varphi^{-1}(B).$$ 
Then, we see that  
$$
\mu^* \pi_* \phi^* \mathcal{V}=
\mu^* \mu_* \gamma_* \gamma^* \varphi^* \mathcal{V} 
=\mu^* \mu_* \varphi^* \mathcal{V} =\varphi^* \mathcal{V}
\text{ on } \Gamma \setminus  (E_\Gamma \cup  \varphi^{-1}(B)) 
$$
by using the diagram (\ref{comm0}) and 
the isomorphism $\mu  \colon \Gamma \setminus E_\Gamma \simeq X \setminus \mu(E_\Gamma)$.
This implies that  
\begin{align*}
\pi^* \mathcal{E}
=\gamma^* \mu^* \big( (\pi_* \phi^* \mathcal{V})^{* *} \big) 
=\gamma^*  \big( (\mu^* \pi_* \phi^* \mathcal{V})^{* *} \big) 
= \phi^* \mathcal{V}
\end{align*}
on  $M \setminus  (\gamma^{-1}(E_\Gamma) \cup  \phi^{-1}(B))$.  
Here, we used the definition of $\mathcal{E}$ and 
the local freeness (reflexivity) of $\varphi^* \mathcal{V}$ on $\Gamma \setminus \varphi^{-1}(B)$. 
\end{proof}

Let us return to the proof of {\hyperref[thm-flat]{Theorem \ref*{thm-flat}}}. 
By assumption, we can easily see that   
the pullback $\pi^* \mathcal{E}$ is 
a numerically flat locally free sheaf on $M$. 
Hence $\pi^* \mathcal{E}$ admits a unique flat connection $\nabla$ 
by \cite[Corollary 3.10 and the discussion after]{Sim92} and \cite[Theorem 1.18]{DPS94}.  
By the definition of $B$ and $\codim B \geqslant 2$, 
it is sufficient to show that $\mathcal{V}$ 
admits a flat connection on $Y_{0} \setminus (B \cup C)$ 
for some Zariski closed set $C$ of codimension $\geqslant 2$. 
For simplicity, we assume that $B=\emptyset$ by replacing $Y$ with $Y\setminus B$ 
(i.e.,\,$\mathcal{V}$ is locally free and $\varphi \colon \Gamma \to Y$ is flat). 
We construct the desired  flat connection of $\mathcal{V}$ on $Y_{0} \setminus C$ 
adopting the following strategy: 
\begin{itemize}
\item[(1)] We demonstrate that the connection $\nabla$ of $\pi^* \mathcal{E}$ 
descends to  a flat connection of $\mathcal{V}$ on a Zariski open set in $Y$, 
by using  $\pi^* \mathcal{E} = \phi^{*}\mathcal{V} $ over $Y \setminus \varphi(E_\Gamma) $. 
\item[(2)] We demonstrate that the flat connection of $\mathcal{V}$ constructed in (1) 
can be extended on $Y_{0} \setminus C$ for some Zariski closed set $C$ of $\codim \geqslant 2$. 
\end{itemize}

\smallskip

\noindent 
(1) We first consider only a neighborhood (in the analytic topology) of a given point in $Y$. 
Therefore, assuming that $Y $ is a (sufficiently small) open ball, 
we take a local frame of $\mathcal{V}$ on $Y$. 
This frame determines the frame of $\pi^* \mathcal{E}$ on $M \setminus \gamma^{-1}(E_\Gamma)$ 
by  {\hyperref[claim2]{Claim \ref*{claim2}}} (recalling that $B=\emptyset$). 
The flat connection $\nabla$ of $\pi^* \mathcal{E}$ can be written as $\nabla=d + \Xi$. 
Here $d$ is the exterior derivative and $\Xi$, 
which is the connection form, is a $(1,0)$-form valued in 
$\End(\pi^* \mathcal{E})$.
We regard $\Xi$ as a matrix of $(1,0)$-forms with respect to the fixed frame. 
By the flatness of $\nabla$, 
the matrix $\Xi$ satisfies the equation  
$$
d \Xi + \Xi \wedge \Xi=0. 
$$
The $(1,1)$-part of the left-hand side is $\dbar \Xi$ (which is zero). 
Hence $\Xi$ can be considered as 
a matrix of holomorphic one-forms 
on $M \setminus \gamma^{-1}(E_\Gamma)$.

We will confirm that $\Xi$ descends to a holomorphic $(1,0)$-form $\Xi'$ 
valued in $\End(\mathcal{V})$ on  $Y_{1}$. 
Here $Y_{1}$ is defined by 
\begin{equation*}
Y_{1}:=\left\{
y \in Y \setminus \varphi (E_{\Gamma})\;\left|\;
  \begin{gathered}
    \text{The fiber } M_y:=\phi^{-1}(y) \text{ is } \\
    \text{smooth and rationally connected}.
  \end{gathered}
\right.
\right\}
\end{equation*}
Since $\psi \colon X \dashrightarrow Y$ is an MRC fibration, 
the fiber $X_y:=\psi^{-1}(y)$ of a general point $y \in Y$ is rationally connected, 
and  so is the fiber $M_y$. 
This indicates that $Y_{1}$ is a non-empty Zariski open set in $Y$. 
It is possible that $\varphi(E) = Y$, but we have $\varphi (E_{\Gamma}) \subsetneq Y$, 
which is the reason why we consider $E_\Gamma$ rather than $E$.
The fiber $M_y$ at $y \in Y_{1}$ has no non-trivial holomorphic one-form. 
Hence, the components of $\Xi$, which are holomorphic one-forms, 
vanish along the fiber $M_y$ at $y \in Y_{1}$. 
Furthermore, 
we have 
$\End(\pi^* \mathcal{E})=\phi^{*} \End (\mathcal{V})$ over $Y  \setminus \varphi(E_\Gamma) $ 
by {\hyperref[claim2]{Claim \ref*{claim2}}}. 
Therefore, we  find a holomorphic section 
$$
\Xi' \in 
H^0(Y_1, \Omega^{1}_Y \otimes \End (\mathcal{V})) 
\text{ such that } \Xi=\phi^* \Xi'. 
$$

\smallskip 
\noindent 
(2)
We will prove that $\Xi' $ can be extended  to a section on $Y_{0} \setminus C$. 
Let $P \subset Y_{0}$ be a divisorial component of $Y \setminus Y_{1}$.
It is sufficient to extend $\Xi' $ through a general point $y \in P$.  
The pullback $\phi^*P$ is not $\pi$-exceptional by the definition of $Y_{0}$. 
Hence, we  find a reduced component $P_{0}$ of  $\phi^*P$ 
by property (e) in {\hyperref[prop-2.2]{Proposition \ref*{prop-2.2}}}. 
Note that $P_{0} \not \subset \gamma^{-1}(E_\Gamma)$
since $\gamma^{-1}(E_\Gamma)$ is $\pi$-exceptional. 
Furthermore, the induced morphism $\phi \colon P_{0} \to P$ is surjective 
since $\phi \colon M \to Y$ is flat over $Y_{0}$ (recalling that $P \subset Y_{0}$). 
Therefore, for a general point $y \in P$, 
we find a point $x \in P_0 \cap M_{y}$ 
such that  
$$\text{
$x \not \in \gamma^{-1}(E_\Gamma)$ and 
$x$ is a smooth point of $\phi$. 
}
$$
The first property shows that 
all the components of the matrix $\Xi=\phi^*\Xi'$ are bounded on a neighborhood of $x$, 
since $\Xi$ is defined not only on $\phi^{-1}(Y_{1})$ 
but also on $M \setminus \gamma^{-1}(E_\Gamma)$. 
The second property yields a local section of $\phi \colon M \to Y$ from $y $ to $x$. 
By pulling back $\Xi=\phi^* \Xi'$ by this local section, we see that 
all the components of the matrix $\Xi'$ are also bounded on a neighborhood of $y$. 
Therefore, the section $\Xi'$ can be extended 
through a general point $y \in P$ by the Riemann extension theorem. 

The extended connection of $\mathcal{V}$ (defined locally on $Y_{0} \setminus C$) 
satisfies the gluing condition as a connection and is  flat from the flatness of $\nabla$. 
Hence, the glued connection determines the flat connection  of $\mathcal{V}$ on $Y_0 \setminus C$. 
\end{proof}

\begin{rem}
\label{rmk_thm-flat}
\begin{itemize}
\item[(a)] The key point in this proof is that a general fiber of $\phi \colon M \to Y$ is rationally connected 
(in particular, it has no holomorphic one-form). Furthermore, we implicitly used the fact that $\phi \colon M \to Y$ has connected fibers.
\item[(b)] By the proof of {\hyperref[thm-flat]{Theorem \ref*{thm-flat}}} we see that the flat vector bundle $\calV_m|_{Y_0\setminus  C_m}$ satisfies the condition ($\bullet$) in \hyperref[prop-2.2]{Proposition \ref*{prop-2.2}(d)} for every $m$.
\end{itemize}
 \end{rem}

\section{MRC fibrations of klt pairs with nef anti-log canonical divisor}\label{sec-MRC}

This section is devoted to the proof  of {\hyperref[thm-main]{Theorem \ref*{thm-main}}}. 
To this end, we study MRC fibrations of projective klt pairs with nef anti-log canonical divisor 
by applying {\hyperref[thm-flat]{Theorem \ref*{thm-flat}}} and the theory of foliations.

\subsection{Local constancy of algebraic fiber spaces}

In this subsection, we study the local constancy and relative anti-canonical divisors 
of algebraic fiber spaces. 
The content of this subsection is not only motivated by the proof of  {\hyperref[thm-main]{Theorem \ref*{thm-main}}},  
but also independently interesting. 

The first result (see {\hyperref[thm_rel-anti-nef]{Theorem \ref*{thm_rel-anti-nef}}}) 
gives conditions on singularities and the positivity of  relative anti-canonical divisors 
to  guarantee that algebraic fiber spaces are locally constant, 
which generalizes \cite[Theorem A]{Wang20} (see {\hyperref[cor_rel-anti-nef]{Corollary\ref*{cor_rel-anti-nef}}}).

\begin{theo}
\label{thm_rel-anti-nef}
Let $h \colon V\to W$ be an algebraic fiber space between normal projective varieties with a smooth $W$. 
Let $D$ be an effective $\mathbb{Q}$-divisor on $V$ such that $(V,D)$ is klt and $-(K_{V/W}+D)$ is nef. 
Then $h \colon V\to W$  is a locally constant fibration with respect to $(V,D)$. 
\end{theo}

\begin{cor}\label{cor_rel-anti-nef}
Let $(X,\Delta)$ be a  projective klt pair with the nef anti-canonical divisor $-(K_X+\Delta)$. 
Then, the Albanese map of $X$ is a locally constant fibration with respect to $(X,\Delta)$.
\end{cor}

To prove {\hyperref[thm_rel-anti-nef]{Theorem \ref*{thm_rel-anti-nef}}}, 
we need the following preliminary result: 

\begin{prop}
\label{prop_rel-anti-nef}
In the same setting as in {\hyperref[thm_rel-anti-nef]{Theorem \ref*{thm_rel-anti-nef}}}, 
the following statements hold$:$
\begin{itemize}
\item[\rm(a)] $h$ is flat. 
\item[\rm(b)] $D$ is horizontal with respect to $h$. 
\item[\rm(c)] For every pseudo-effective and $\phi$-big divisor $G$ on $V$, 
the direct image sheaf 
$$
h_\ast\OX_V(q(K_{V/W}+D)+G)
$$ 
is weakly positively  curved for every $q\in\ZZ$. 
\item[\rm(d)] For every $\phi$-big divisor $G$ on $V$ and every $m\in\ZZ_{>0}$, 
we define the Cartier divisor  by 
\[
D_{G,m}:=\frac{1}{r_m}\cdot\text{the Cartier divisor associated with }\det\!h_\ast\OX_V(mG), 
\]
where $r_m:=\rank h_\ast\OX_V(mG)$. Then  $G-h^\ast\!D_{G,1}$ is pseudo-effective.
\item[\rm(e)] Let $A$ be a sufficiently ample divisor on $V$ such that 
$$
\Sym^k\!\Coh^0(V_w,\OX_{V_w}(A))\twoheadrightarrow\Coh^0(V_w,\OX_{V_w}(kA))
$$ 
for a general fiber $V_w$. 
Then $D_{A,m}\equiv_{{\rm{num}}}mD_{A,1}$ holds for every $m\in\ZZ_{>0}$.  
\end{itemize}
\end{prop}
\begin{proof}
The proof is obtained from 
considering a resolution of the singularities of $V$ and 
applying the same argument as in \cite{CCM19,Wang20}. 
Specifically, (a) is a direct consequence of \cite[Lemma 43.14]{Wang20}; 
(b) follows  from the proof of \cite[Proposition 3.1(b)]{Wang20} or \cite[Theorem 1.3]{CCM19}; 
(c) follows from the proof of \cite[Lemma 3.4]{CCM19} or \cite[Proposition 3.2]{Wang20}; 
(d) follows from the proof of \cite[Lemma 3.5]{CCM19} or \cite[Proposition 3.3]{Wang20}; 
and (e) follows from \cite[Proposition 3.6]{CH19} or \cite[Proposition 3.5]{Wang20}. 
\end{proof}

\begin{proof}[Proof of {\hyperref[thm_rel-anti-nef]{Theorem \ref*{thm_rel-anti-nef}}}]
Let $A$ be a sufficiently ample divisor on $V$ satisfying 
{\hyperref[prop_rel-anti-nef]{Proposition \ref*{prop_rel-anti-nef}(e)}}. 
By  {\hyperref[prop_rel-anti-nef]{Proposition \ref*{prop_rel-anti-nef}(e)}},  
we may assume that $D_{A,1}$ is a  $\ZZ$-divisor after replacing $A$ with its multiple.
By replacing $A$ with $A-\phi^\ast\!D_{A,1}$, we have the following:  
\begin{itemize}
\item $A$ is pseudo-effective on $V$ by {{\hyperref[prop_rel-anti-nef]{Proposition \ref*{prop_rel-anti-nef}(d)}}}. 
\item $A$ is $h$-very ample. 
\item $\Sym^k\!\Coh^0(V_w,\OX_{V_w}(A))\twoheadrightarrow\Coh^0(V_w,\OX_{V_w}(kA))$. 
\item $D_{A,1}\sim 0$. 
\end{itemize}
Then, the sheaf $h_\ast\OX_V(mA)$ is reflexive 
since $h$ is flat by {\hyperref[prop_rel-anti-nef]{Proposition \ref*{prop_rel-anti-nef}(a)}} 
and weakly positively curved  for any $m\in\ZZ_{>0}$ 
by {\hyperref[prop_rel-anti-nef]{Proposition \ref*{prop_rel-anti-nef}(c)}}. 
Moreover, since $D_{A,m}\equiv_{{\rm{num}}} mD_{A,1}\sim0$ holds, 
we can conclude that $h_\ast\OX_V(mA)$  is a numerically flat vector bundle on $W$ 
by  \cite[Proposition 2.7]{CCM19} or \cite[\S 1, Corollary of Main Theorem]{Wu20}. 

It sufficient to prove that $h_\ast\OX_V(mA-pD)$ is a numerically flat vector bundle 
for every $m\in\ZZ_{>0}$ and for some $p\in\ZZ_{>0}$ rendering $pD$ a $\ZZ$-divisor, 
which follows from {\hyperref[prop_flat-lcf]{Proposition \ref*{prop_flat-lcf}}} 
and {\hyperref[rmk_prop_flat-lcf]{Remark \ref*{rmk_prop_flat-lcf}(b)}}. 
To prove this, for a log resolution $\mu \colon V'\to V$ of $(V,D)$, 
we write as 
\[
K_{V'}+D'\sim_{\QQ}\mu^\ast(K_V+D)+E 
\]
with $(V',D')$ being a klt pair and $E$ being an effective $\mu$-exceptional $\mathbb{Q}$-divisor. 
We fix $p\in\ZZ_{>0}$ such that both $pD$ and $pD'$ are $\ZZ$-divisors. 
For a sufficiently large $q$ (noting that $q$ may depend on $m$), we write as 
\[
qE+m\mu^\ast\!A-pD'\sim qK_{V'/W}-q\mu^\ast(K_{V/W}+D)+(q-p)D'+m\mu^\ast\!A. 
\]
By the klt condition, we have 
$$
\scrJ(h_{D'}^{\frac{q-p}{q}}\cdot\mu^\ast h_A^{\frac{m}{q}})\simeq\OX_{V'} 
$$ 
for $q \gg1$, where $h_A$ is a singular Hermitian metric on $A$ with semi-positive curvature 
and $h_{D'}$ is the singular Hermitian metric associated with the effective $\QQ$-divisor $D'$. 
Hence, by \cite[Corollary 2.1.3]{Wang20}, the direct image sheaf $(\mu\circ h)_\ast\OX_{V'}(qE+m\mu^\ast\!A-pD)$ is weakly positively curved. 
Moreover, by 
$
\mu_\ast(m\mu^\ast\!A-pD')=mA-pD,
$
we obtain 
\[
\mu_\ast\OX_{V'}(qE+m\mu^\ast\!A-pD')\simeq\left(\mu_\ast\OX_V(m\mu^\ast\!A-pD')\right)^{**}\simeq\OX_V(mA-pD)
\]
for $q \gg1$ (cf.\,\cite[Theorem 1.3.1 and Lemma 1.3.2]{Wang-thesis}). 
Consequently, we see that $h_\ast\OX_V(mA-pD)$ is weakly positively  curved and $D_{mA-pD,1}$ is pseudo-effective. 
Meanwhile, we have $D_{mA-pD,1}\equiv_{{\rm{num}}} 0$ 
since $D$ is effective and $D_{A,m}\equiv_{{\rm{num}}} 0$. 
Furthermore, the sheaf $h_\ast\OX_V(mA-pD)$ is reflexive since $h$ is flat and $\OX_V(mA-pD)$ is reflexive. 
Thus, by \cite[Proposition 2.7]{CCM19} or \cite[\S 1, Corollary of Main Theorem]{Wu20}, 
we see that $h_\ast\OX_V(mA-pD)$ is a numerically flat vector bundle for every $m\in\ZZ_{>0}$.     
\end{proof}

We next prove the three lemmas (Lemmas \ref{lemme_nef-big}, \ref{lemme_lcf-birational}, and \ref{lemma_Cartesian-lcf}), 
which are needed for comparing MRC fibrations of a given variety $X$ to 
those of quasi-\'etale covers or birational models of $X$.

\begin{lemm}
\label{lemme_nef-big}
Let $h \colon V\to W$ be a locally constant fibration between normal projective varieties such that the fiber $F$ has 
vanishing irregularity. 
Assume that there is an effective $\mathbb{Q}$-divisor $D$ on $V$ such that $(V,D)$ is klt and $-(K_{V/W}+D)$ is nef. 
Then, any nef and big divisor $B$ on $V$ can be written as 
$B=B_0+h^\ast B_W$, 
where $B_{0}$ is a nef $\QQ$-divisor on $V$ and $B_{W}$ is a nef and big  $\QQ$-divisor  on $W$.  
\end{lemm}
\begin{proof}
Let $m$ be a sufficiently divisible  integer.
By {\hyperref[lemma_local-const-lb-decomp]{Lemma \ref*{lemma_local-const-lb-decomp}}}, 
we can write as $B=B_0+h^\ast B_W$ such that $h_\ast\OX_V(mB_0)$ is a flat vector bundle 
for every $m$ sufficiently divisible. 
We will prove that $B_0$ is nef and $B_W$ is nef and big. 
To this end, we may assume that $W$ is smooth after $W$ is replaced with a resolution of singularities and 
$(V,D)$ is replaced with the induced fiber product. 
Indeed, after the base change, 
the klt condition is preserved, 
the fibration $h$ remains a locally constant fibration, 
both $B_0$ and $B_W$ are obtained from the pullbacks.

We can apply {\hyperref[thm_rel-anti-nef]{Theorem \ref*{thm_rel-anti-nef}}} since $W$ is smooth. 
Hence $h$ is a locally constant fibration also with respect to the pair $(V,D)$. 
Moreover, we see that $B_{0}$ is pseudo-effective by {\hyperref[prop_rel-anti-nef]{Proposition \ref*{prop_rel-anti-nef}(d)}} 
and $E_m:=h_\ast\OX_{V}(mB_0)$ is weakly positively  curved 
by {\hyperref[prop_rel-anti-nef]{Proposition \ref*{prop_rel-anti-nef}(b)}}.
Therefore 
$
E_m=h_\ast\OX_V(mB_0)
$
is a numerically flat vector bundle 
since $h_\ast\OX_{V}(mB_0)$ is a flat vector bundle. 

We next show that $h_\ast\OX_V(mB-mD)$ is a nef vector bundle. 
The tangent sheaf $T_V$ is decomposed into 
$T_V\simeq T_{V/W}\oplus\calE$ with $\calE\simeq h^\ast T_W$ (e.g.,\,see \cite[Remark 21.27]{Wang20})
since $h$ is a locally constant fibration, which implies that   $\calE$ is locally free. 
Then, since $(V,D)$ is a log canonical pair, a functorial resolution $\mu \colon V'\to V$ satisfies  
$T_{V'}\simeq \mu^\ast\calE\oplus T_{V'/W}$ by \cite[Lemma 5.10]{Druel18a}. 
(Note that \cite[Lemma 5.10]{Druel18a} holds for any klt pairs since it depends only on \cite{GKKP11}). 
Consequently, the fibration  $h\circ\mu \colon V'\to W$ is a locally constant fibration by the classical Ehresmann theorem 
(cf.\,\cite[3.17.Theorem]{Hor07} and \cite[\S V.3, Theorem 1 and Theorem 3, pp.\,91-95]{CLN85}). 
Meanwhile, since $B-(K_{V/W}+D)$ is nef and big by assumption, 
the direct image sheaf 
$$
(h\circ\mu)_\ast\OX_V(mK_{V'/W}+m\mu^\ast(B-K_{V/W}-D))
$$ 
is a nef vector bundle for $m\in\ZZ_{>0}$ sufficiently divisible, 
which essentially follows from a variant of \cite[Th\'eor\`eme 1]{Mou97} or \cite[Theorem 1.4 and \S 5]{Fuj16}. 
Indeed, we obtain the nefness by replacing \cite[Theorem 5.1]{Fuj16} or \cite[Theorem 1.4]{PS14} 
with \cite[Variant 1.5]{PS14} in \cite[\S 5, Proof of Theorem 1.4]{Fuj16}. 
Moreover, by \cite[Lemma 1.3.2]{Wang-thesis}, we obtain
\[
(h\circ\mu)_\ast\OX_{V'}(mK_{V/Y'}+m\mu^\ast(B-K_{V/W}-D))\simeq h_\ast\OX_V(mB-mD) 
\]
since the left-hand side is locally free (and thus reflexive). 
Hence $h_\ast\OX_V(mB-mD)$ is a  nef vector bundle.

The decomposition of $B-D$ given by {\hyperref[lemma_local-const-lb-decomp]{Lemma \ref*{lemma_local-const-lb-decomp}}} 
should be $B-D=(B_0-D)+h^\ast\!B_W$ since $h \colon V\to W$ is a locally constant fibration with respect to the pair $(V, D)$.
Therefore  $E_{m,D}:=h_\ast\OX_V(mB_0-mD)$ is a flat vector bundle. Furthermore, by the projection formula, we have
\[
h_\ast\OX_V(mB-mD)\simeq E_{m,D}\otimes\OX_W(mB_W), 
\]
which indicates that $B_W$ is nef.
 
Since $B_{W}$ is nef, by \cite[(7.5) Corollary, p.\,52]{Dem01} or \cite[(7.5) Corollary, p.\,66]{Dem10}, 
there is an ample divisor $H$ on $W$ (independent of $m$) such that $H+mB_{W}$ is very ample for every $m$. For a general hypersurface $H_{m,1}\in|H+mB_{W}|$, we have the exact sequence
\[
0 \to E_m \otimes \OX_{W}(-H) \to E_m \otimes \OX_{W}(mB_{W}) \to  E_m \otimes \OX_{W}(mB_{W})  |_{H_{m, 1}}\to 0. 
\]
Since $E_m$ is numerically flat, we have $\Coh^0(W,E_m\otimes\OX_W(-H))=0$. Hence, we see that 
\[
\dimcoh^0(W,E_m\otimes\OX_{W}(mB_{W}))\leqslant \dimcoh^0(H_{m,1}, \left(E_m\otimes\OX_{W}(mB_{W})\right)|_{H_{m,1}}).
\]
By repeating this process, we obtain 
\begin{align*}
\dimcoh^0(V, \OX_V(mB)) &=\dimcoh^0(W, E_m \otimes\OX_{W}(mB_{W})) \\
&\leqslant \sharp(H_{m,1} \cap H_{m,2} \cap  H_{m,\dim Y}) \cdot \rank\!E_m\\
&\leqslant  (mB_{W}+H)^{\dim Y} \cdot \dimcoh^0(F, mB_0|_{F}).
\end{align*} 
By considering $\lim_{m\to+\infty}( \bullet/m^{\dim X})$ 
(cf.\,\cite[\S 1.2.B, Example 1.2.36, p.\,38-39, \S 2.2.C, Definition 2.2.31 and (2.9), p.\,148, vol.I]{Laz04}), we obtain 
\[
0<\volume_{V}(B) \leqslant C\cdot\volume_{W}(B_{W}) \cdot \volume_{F}(B_0|_{F}) 
\]
for some constant $C$ (depending only  on the dimensions of $X$ and $Y$), which implies that $B_{W}$ is big.

It remains to show that $B_0$ is nef. 
Let $A$ be a very ample divisor on $W$ so that it induces an embedding $W\hookrightarrow\PP\Coh^0(W,\OX_{W}(A))$. Since $E_m$ is a flat vector bundle, $E_m\otimes\OX_{W}(kA)$ is Nakano positive for every $m,k\in\ZZ_{>0}$ (cf.\,\cite[(3.9) Definition, p.\,24]{Dem01} or \cite[(3.9) Definition, p.\,28]{Dem10}). Hence, by the Nakano vanishing theorem (cf.\,\cite[(4.9), p.\,30]{Dem01} or \cite[(4.9), p.\,35]{Dem10}), 
we obtain  
\[
\Coh^i(W,E_m\otimes\OX_{W}((d+1-i)A))=0
\]
for every $i>0$, where $d=\dim\!Y$. 
Consequently, the vector bundle $E_m$ is $(d+1)$-regular with respect to $A$. 
By \cite[Theorem 1.8.3, pp.\,99-100]{Laz04}, we see that  
\[
E_m\otimes\OX_{W}((d+1)A)\simeq h_\ast\OX_{V}(mB_0+(d+1)h^\ast\!A)
\]
is globally generated on $W$. Moreover, since $-K_{V/W}$ is nef and $B_0$ is $h$-relatively big and nef, 
the relative base point freeness shows that $B_0$ is $h$-globally generated. 
Hence $mB_0+(d+1)h^\ast\!A$ is globally generated.
 By letting $m\to+\infty$, we can conclude that $B_0$ is nef.  
\end{proof}

\begin{lemm}
\label{lemme_lcf-birational}
Let $V$, $W$, $V'$ and $W'$ be normal projective varieties satisfying following commutative diagram$:$ 
\begin{center}
\begin{tikzpicture}[scale=2.5]
\node (A) at (0,0) {$W$.};
\node (B) at (0,1) {$V$};
\node (A1) at (-1,0) {$W'$};
\node (B1) at (-1,1) {$V'$};
\path[->,font=\scriptsize,>=angle 90]
(B) edge node[right]{$f$} (A)
(B1) edge node[left]{$f'$} (A1)
(B1) edge node[above]{$\beta_V$} (B)
(A1) edge node[below]{$\beta_W$} (A);
\end{tikzpicture}  
\end{center}
Here $\beta_V$ and $\beta_W$ are birational morphisms and $f$ and $f'$ are fiber spaces. 
Assume that $f'$ is a locally constant fibration whose fiber $F'$ is simply connected. 
Assume  $W$ has klt singularities and that there is an effective divisor $\Delta'$ on $V'$ such that $(V',\Delta')$ is klt and that $-(K_{V'/W'}+\Delta')$ is nef. Then $f$ is a locally trivial fibration. 
\end{lemm}

\begin{proof}
Let $p_W \colon \Unv W\to W$ (resp.\,$p_{W'} \colon \Unv{W'}\to W'$) be the universal cover of $W$ (resp.\,of $W'$), we will prove the stronger statement that the pullback of $V$ over $\Unv{W}$ splits into a product $\Unv{W}\times F$, 
where $F$ is the general fiber of $f$ (which implies that $f$ is a locally trivial fibration).

Since $(V',\Delta')$ is klt and since $f'$ is a locally constant fibration, 
we see that $W'$ has klt singularities. Moreover, since $W$ has also klt singularities, we have $\pi_1(W_1)\simeq\pi_1(W)$ from \cite[Theorem 1.1]{Tak03}. 
Then, the morphism $\beta_W\circ p_{W'}$ lifts to a projective morphism $\tilde\beta_W \colon \Unv{W'}\to \Unv W$ by the map lifting lemma, and the bottom square of the commutative diagram below is Cartesian. In addition, we have $(\tilde\beta_W)_\ast\OX_{\Unv{W'}}\simeq\OX_{\Unv W}$. 

\begin{center}
\begin{tikzpicture}[scale=2.0]
\node (A) at (0,0) {$W$};
\node (B) at (0,1) {$V$};
\node (A1) at (-1,0) {$W'$};
\node (B1) at (-1,1) {$V'$};
\node (A2) at (-2.4,-0.5) {$\Unv{W'}$};
\node (B2) at (-2.4,1.5) {$V'\underset{W'}{\times}\Unv{W'}\simeq\Unv{W'}\times F'$};
\node (S) at (-1.7,0.5) {$\square$};
\node (A') at (1.2,-0.5) {$\Unv W.$};
\node (B') at (1.2,1.5) {$V\underset{W}{\times}\Unv W$};
\node (S') at (0.6,0.5) {$\square$};
\node (Q) at (-0.5,-0.25) {$\square$};
\node (C) at (2.4,2) {$\Unv W\times F'$}; 
\path[->,font=\scriptsize,>=angle 90]
(B) edge node[right]{$f$} (A)
(B1) edge node[right]{$f'$} (A1)
(B1) edge node[above]{$\beta_V$} (B)
(A1) edge node[above]{$\beta_W$} (A)
(A2) edge node[below right]{$p_{W'}$} (A1)
(B2) edge node[above right]{$p_{V'}$} (B1)
(B2) edge node[left]{$\tilde f'=\pr_1$} (A2)
(B') edge node[right]{$\tilde f$} (A')
(B') edge node[above left]{$p_V$} (B)
(A') edge node[below left]{$p_W$} (A)
(A2) edge node[below]{$\tilde\beta_W$} (A')
(B2) edge node[above]{$\tilde\beta_V$} (B')
(B2) edge[bend left] node[above left]{$\tilde\beta_W\times\id_{F'}$} (C)
(C) edge[bend left] node[right]{$\pr_1$} (A')
(C) edge node[above left]{$\exists\,\beta$} (B');
\end{tikzpicture}
\end{center}

Let us take the fiber product $V\times_W\Unv W$ 
equipped with the natural morphisms $p_V \colon V\times_W\Unv W\to V$ and $\tilde f \colon V\times_W\Unv W\to \Unv W$. 
Note that $p_{V'}$ is  the universal cover of $V'$ 
since $F'$ is simply connected. Hence $\beta_V\circ p_{V'}$ lifts to the morphism 
$$
\tilde\beta_V \colon \Unv{W'}\times F'\to V\times_W\Unv W.
$$ 
Furthermore, the morphism $\tilde\beta_V$ is projective and we have $(\tilde\beta_V)_\ast\OX_{\Unv{W'}\times F'}\simeq\OX_{V\times_W\Unv W}$ since both $\tilde f'$ and $\tilde\beta_W$ are  projective and have connected fibers. 

Take an ample divisor $A$ on $V$.  
Then $B:=\beta_V^\ast A$ is nef and big over $V'$, 
we can apply {\hyperref[lemme_nef-big]{Lemma \ref*{lemme_nef-big}}}, 
and write $B=B_0+(f')^\ast B_{W'}$ such that $B_0$ is nef and $B_{W'}$ is nef and big. 
Moreover $B_{W'}$ is $\beta_W$-relatively numerically trivial. 
To see this, we take an integral curve $C_0$ on $W'$ that is contracted by $\beta_W$, and let $C_1$ be a curve on $V'$ such that $g(C_1)=C_0$.  
Then $C_1$ must be contracted by $\beta_V$. 
Indeed, otherwise $\beta_V(C_1)$ would be contained in some fiber of $f$ and thus $C_1$ would be contracted by $f'$,  which is absurd. 
Therefore we have
\[
B_0\cdot C_1+B_{W'}\cdot C_0=B\cdot C_1=\beta_V^\ast A\cdot C_1=0. 
\]
Since $B_0$ and $B_{W'}$ are nef, we must have $B_{W'}\cdot C_0=0$, i.e.,\,$B_{W'}$ is $\beta_W$-relatively numerically trivial. 
Hence, we can find a big and nef divisor $A_W$ on $W$ such that $B_{W'}\equiv\beta_W^\ast A_W$.

By applying \cite[Lemma 4.2.4]{Wang-thesis} we see that $\tilde\beta_V$ factorizes through $\tilde\beta_W\times\id_{F'}$, i.e.,\,there is a projective morphism $\beta \colon \Unv W\times F'\to V\times_W\Unv W$ such that $\tilde\beta_V=\beta\circ(\tilde\beta_W\times\id_{F'})$. 
Furthermore, we have $\beta_\ast\OX_{\Unv W\times F'}\simeq\OX_{V\times_W\Unv W}$. 
Hence $\beta$ must be birational. 
Then, from \cite[Lemma 4.2.3]{Wang-thesis}, we conclude that $\tilde f$ induces a decomposition of $V\times_W\Unv W$ into a product $\Unv W\times F$. In particular, the fibration $f$ is  locally trivial. 
\end{proof}

\begin{lemm}
\label{lemma_Cartesian-lcf}
Let $h \colon V\to W$ be an algebraic fiber space between projective varieties, and $D$ be an effective $\mathbb{Q}$-divisor on $V$. 
Assume that $W$ has klt singularities and that there is a projective birational morphism $\mu \colon W'\to W$ such that the base change morphism 
$h' \colon V':=V\times_W W'\to W'$ is a locally constant fibration with respect to $(V',D')$ where $D'$ is the pullback of $D$. 
Then $h$ is also a locally constant fibration with respect to $(V,D)$.   
\end{lemm}

\begin{proof}
By \cite[Theorem 1.1]{Tak03}, we have 
$$\pi_1(W')\xrightarrow[\simeq]{\mu_\ast}\pi_1(W).$$
Then, by considering the pullback family over the universal cover of $W$ and $W'$ respectively and applying \cite[Lemma 4.2.3, Lemma 4.2.4]{Wang-thesis}, we can easily deduce that $h$ is a locally constant fibration with respect to $(V,D)$. 
The argument is quite similar (and in fact much simpler than) to that of \hyperref[lemme_lcf-birational]{Lemma \ref*{lemme_lcf-birational}}.
\end{proof}

\subsection{Splitting of tangent sheaves}
\label{ss_MRC_tangent}
In this subsection, we 
deduce the splitting of tangent sheaves for projective klt pairs with anti-log canonical divisor 
from {\hyperref[thm-flat]{Theorem \ref*{thm-flat}}}. 

\begin{theo}
\label{thm_splitting}
Let $(X,\Delta)$ be a projective klt pair with nef anti-canonical divisor $-(K_X+\Delta)$.
Assume that $X$ is maximally quasi-\'etale. 
Then, the MRC fibration of $X$ induces a splitting of the tangent sheaf $T_{X}$ of $X$: 
$$
T_X\simeq\calF\oplus\calG 
$$
such that 
\begin{itemize}
\item[$\bullet$] $\calF$ is an algebraically integrable foliation 
whose general leaves are rationally connected fibers of MRC fibrations of $X$$;$ 
\item[$\bullet$] $\calG$ is a $($possibly singular$)$ foliation whose canonical divisor $K_{\calG}\sim_{\QQ}0$. 
\end{itemize}
\end{theo}
Here a (singular) foliation $\mathcal{H}$ denotes 
a saturated subsheaf $\mathcal{H} \subset T_{X}$ that is closed under the Lie bracket. 
By the normality of $X$, 
the restriction $\mathcal{H}|_{X_{\reg}}$ is a genuine foliation in the usual sense, 
and the leaves of $\mathcal{H}$ are thus defined. 
A foliation is said to be {\it algebraically integrable} if its general leaf is Zariski open in its closure (cf.\,\cite{GM89}). 
(See \cite[\S 2.4]{Wang-thesis} for a general account.)

For  the proof of {\hyperref[thm_splitting]{Theorem \ref*{thm_splitting}}}, 
we first confirm the following key lemma, which states that the desired splitting holds in the $\QQ$-factorial terminal case.

\begin{lemm}
\label{lemma_splitting}
Let $X$ be  as in {\hyperref[thm_splitting]{Theorem \ref* {thm_splitting}}}, and assume in addition that $(X,\Delta)$ is a $\QQ$-factorial terminal pair. 
Then, the MRC fibration of $X$ induces a splitting of the tangent sheaf of $X$ with the same properties as those in the statement of {\hyperref[thm_splitting]{Theorem \ref*{thm_splitting}}}.  
\end{lemm}
\begin{proof}
This proof is based on \cite[\S 3.C, Proof of Theorem 1.2]{CH19} 
and \cite[Step 3 \& Step 4 of the Proof of Theorem 54.1]{Wang20}. 
For the reader's convenience, we briefly recall the proof. 
Let the conditions be the same as those in {\hyperref[setting]{Setting \ref*{setting}}}. 
For a very ample divisor $A$ on $X$, 
we set $G:=\pi^\ast A+cE$ and 
\[
D_{A,c,m}:=\frac{1}{r_m}\cdot\text{the Cartier divisor associated with } \det\!\phi_\ast\OX_M(mG),
\]
where $r_m:=\rank\!\phi_\ast\OX_M(mG)$. 
By replacing $M$ with its successive blow-ups, 
we assume that the $\phi$-relative base locus of $G$ is a divisor. Then $G$ can be written as follows: 
$$
G=G_{\free}+G_{\base}, 
$$
where $G_{\base}$ is the $\phi$-relative fixed part of the linear system $|G|$ and 
$G_{\free}:=G-G_{\base}$ is the $\phi$-relatively generated part. 
The adjunction morphism now admits a factorization 
\[
\phi^\ast\phi_\ast\OX_M(G)\twoheadrightarrow\OX_M(G_{\free})\hookrightarrow\OX_M(G),
\]
which pushes down to $Y$ to give the morphisms
\[
\phi_\ast\OX_M(G)\to\phi_\ast\OX_M(G_{\free})\hookrightarrow\phi_\ast\OX_M(G).
\]
By construction, the composition morphism is the identity, 
and the inclusion $\phi_\ast\OX_M(G_{\free})\hookrightarrow\phi_\ast\OX_M(G)$
 is thus an isomorphism. 
Then,  the surjection $\phi^\ast\phi_\ast\OX_M(G_{\free})\twoheadrightarrow\OX_M(G_{\free})$ induces the morphism 
\[
\pi_G\colon M\to\PP(\phi_\ast\OX_M(G_{\free})) \text{ such that }\OX_M(G_{\free})=\pi_G^\ast\OX_{\PP(\phi_\ast\OX_M(G_{\free}))}(1).
\]
We define $X_G$ to be the image $\pi_G (X)$ with the induced morphism $\psi_G \colon X_G\to Y$, 
then we obtain the following commutative diagram:

\begin{center}
\begin{tikzpicture}[scale=2.5]
\node (A) at (0,0) {$Y$.};
\node (B) at (0,1) {$M$};
\node (C) at (-1,1) {$X$};
\node (D) at (1,1) {$X_G$};
\node (E) at (2,1) {$\PP(\phi_\ast\OX_M(G_{\free}))$};
\path[->,font=\scriptsize,>=angle 90]
(B) edge node[left]{$\phi$} (A)
(B) edge node[above]{$\pi$} (C)
(B) edge node[above]{$\pi_G$} (D)
(D) edge node[below right]{$\psi_G$} (A)
(E) edge [bend left=20] node[below right]{$p$} (A);
\path[dashed, ->,font=\scriptsize,>=angle 90]
(C) edge node[below left]{$\psi$} (A);
\path[right hook->,font=\scriptsize,>=angle 90]
(D) edge (E);
\end{tikzpicture}
\end{center}

Since $(X,\Delta)$ is a $\QQ$-factorial terminal pair and $X$ is maximally quasi-\'etale, 
the sheaf $\calV_m:=\phi_\ast\OX_X(mG)\otimes\OX_Y(-mD_{A,c,1})$
is a flat vector bundle over $Y_0\setminus  C_m$ for every $m$ by {\hyperref[thm-flat]{Theorem \ref*{thm-flat}}}, 
where $\codim C_m\geqslant 2$. We set
\[
\calW_m:=\psi_{G\ast}\OX_{X_G}(m)\otimes\OX_Y(-mD_{A,c,1}). 
\]
Then, by {\hyperref[rmk_thm-flat]{Remark \ref*{rmk_thm-flat}(b)}} and by applying the same argument as \cite[Lemma 5.3]{Wang20}, 
we see that $\calW_m\simeq\calV_m$ over $Y_0$.  
As a consequence, for every $m$,  the sheaf $\calW_m$ is a flat vector bundle on $Y_0\setminus  C_m$ and satisfies the condition ($\bullet$) in {\hyperref[prop-2.2]{Proposition \ref*{prop-2.2}(d)}}. Therefore $\SheafHom[](\Sym^m\calW_1,\calW_m)$ is a flat vector bundle on $Y_0\setminus  C_m$ satisfying the condition ($\bullet$) in {\hyperref[prop-2.2]{Proposition \ref*{prop-2.2}(d)}}. 
Thus, each of the global sections is parallel. In particular, this implies that the compatibility condition (2) in {\hyperref[prop_flat-lcf]{Proposition \ref*{prop_flat-lcf}}} is satisfied. 
Combining this result with {\hyperref[rmk_prop_flat-lcf]{Remark \ref*{rmk_prop_flat-lcf}(c)}}, 
we conclude that the morphism $\psi_G \colon X_G\to Y$ is a locally constant fibration over $Y_0$ with the typical fiber $F$ 
(up to replacing $Y_0$ with a Zariski open set whose complement is of codimension $\geqslant 2$). 
In particular, the tangent sheaf $T_{\phi_G\inv(Y_0)}$ splits into two foliations. 
Indeed, by the definition of locally constant fibrations, 
the fundamental group $\pi_{1}(Y_{0})$ diagonally acts on $\Unv{Y_{0}} \times F$. 
In addition to the normality of $Y$, 
the natural splitting 
$$
T_{\Unv{Y_{0}} \times F} =\pr_{1}^{*} T_{\Unv{Y_{0}} } \oplus \pr_{2}^{*} T_{F }
$$
induces a splitting of $T_{\phi_G\inv(Y_0)}$. 
Meanwhile, by \cite[Lemma 5.5]{Wang20} (noting that the assumption $\pi_1(X_{\reg})=\{1\}$ is not needed), 
any divisorial component of the exceptional locus 
$$
\psi_{G}|_{{\phi^{-1}_{0}}(Y_{0})} \colon {\phi^{-1}_{0}}(Y_{0}) \to \psi_{G}^{-1}(Y_{0})
$$ 
is contained in the $\pi$-exceptional locus $E$.
Hence, the splitting of $T_{\phi_G\inv(Y_0)}$ induces a splitting $T_X\simeq\calF\oplus\calG$  of $T_X$. 
By construction,  general leaves of $\calF$ are fibers of the MRC fibration $\psi \colon X \dashrightarrow Y$ 
(in particular, they are rationally connected). 
Thus, by {\hyperref[prop-2.2]{Proposition \ref*{prop-2.2}(a)}}, 
we find 
$$
K_{\calG}|_{\pi(\phi\inv(Y_0)\setminus  E)}\sim_{\QQ}0. 
$$
This fact indicates that $K_{\calG}\sim_{\QQ}0$ 
since both $\pi(E)$ and $\pi(\phi^{-1}(Y \setminus Y_0))$
are of codimension $\geqslant 2$ in $X$ according to {\hyperref[prop-2.2]{Proposition \ref*{prop-2.2}(c)}}. 
\end{proof}

We now prove {\hyperref[thm_splitting]{Theorem \ref*{thm_splitting}}}. 

\begin{proof}[Proof of {\hyperref[thm_splitting]{Theorem \ref*{thm_splitting}}}]
Considering a $\QQ$-factorial terminal model, 
we will reduce the general case to the case of {\hyperref[lemma_splitting]{Lemma \ref*{lemma_splitting}}}. 
Applying \cite[Corollary 1.4.3]{BCHM10}, 
we take a $\QQ$-factorial terminal model $g \colon X^{\terminal}\to X$ of $(X, \Delta)$.  
By construction, 
there is an effective $\QQ$-divisor $\Delta^{\terminal}$ on $X^{\terminal}$ such that 
\[
K_{X^{\terminal}}+\Delta^{\terminal}\sim_{\QQ}g^\ast(K_X+\Delta). 
\]
Hence, the anti-log canonical divisor $-(K_{X^{\terminal}}+\Delta^{\terminal})$ is nef. 
By {\hyperref[thm-mqe]{Theorem \ref*{thm-mqe}}} (2), the variety $X^{\terminal}$ is still maximally quasi-\'etale, and we thus apply {\hyperref[lemma_splitting]{Lemma \ref*{lemma_splitting}}} to get the splitting $T_{X^{\terminal}}\simeq\calF^{\terminal}\oplus\calG^{\terminal}$ with the properties in {\hyperref[thm_splitting]{Theorem \ref*{thm_splitting}}}. Since $g \colon X^{\terminal}\to\bar X$ is birational, the splitting $T_{X^{\terminal}}\simeq\calF^{\terminal}\oplus\calG^{\terminal}$ induces 
a splitting $T_X\simeq\calF\oplus\calG$ with the properties in {\hyperref[thm_splitting]{Theorem \ref*{thm_splitting}}}.
\end{proof}

\subsection{Case of \texorpdfstring{$\QQ$}{text}-factorial terminal pairs with splitting tangent sheaf}
\label{ss_MRC_terminal}
The purpose of this subsection is to confirm the following theorem, 
which asserts that the conclusion of {\hyperref[thm-main]{Theorem \ref*{thm-main}}} holds for $\QQ$-factorial terminal pairs with the splitting tangent sheaf as in {\hyperref[thm_splitting]{Theorem \ref*{thm_splitting}}}.

\begin{theo}\label{thm-main-term}
Let $(X,\Delta)$ be a $\QQ$-factorial terminal pair with nef anti-log canonical divisor $-(K_{X}+\Delta)$. 
Assume that $X$ is maximally quasi-\'etale. 
Then $X$ admits a holomorphic MRC fibration $f \colon X \to Y$ 
such that $f$ is a locally constant fibration with respect to $(X,\Delta)$ and $Y$ is a projective variety with terminal singularities and $K_{Y}\sim_{\QQ}0$. 
\end{theo}

\begin{rem}\label{rem-ex}
\begin{itemize}
\item This theorem does not directly imply {\hyperref[thm-main]{Theorem \ref*{thm-main}}} even in the $\QQ$-factorial terminal case. 
The point is that to make $X$ maximally quasi-\'etale, we need to take a quasi-\'etale cover (see {\hyperref[thm-mqe]{Theorem \ref*{thm-mqe}(1)}}), but taking quasi-\'etale covers destructs $\QQ$-factoriality, and thus prevents us from applying {\hyperref[thm-main-term]{Theorem \ref*{thm-main-term}}}. 
Therefore, the argument in {\hyperref[ss_MRC_klt]{Subsection \ref*{ss_MRC_klt}}} is necessary for the proof of {\hyperref[thm-main]{Theorem \ref*{thm-main}}} even in the $\QQ$-factorial terminal case. 
\item In the above theorem, the condition that $X$ is maximally quasi-\'etale is indispensable. Indeed, 
the theorem does not hold without this condition, even when we assume the splitting of the tangent sheaf. 
Consider the following example: Let $A$ be an abelian variety of dimension $\geqslant 2$, and $\sigma$ be an involution of $A$, consider the group $\langle\sigma\rangle\simeq\ZZ/2\ZZ$ acting on $A\times\PP^1$ by $\sigma\cdot(a,[z,w])\mapsto(\sigma(a),[z,-w])$, and let $X$ be the quotient of $A\times\PP^1$ by this action. 
Then $X$ has $\QQ$-factorial terminal singularities with $-K_X$ nef, and we get a fiber space $f \colon X\to Y$ where $Y:=A/\langle\sigma\rangle$.
Moreover, the natural splitting of $T_{A\times\PP^1}$ induces a splitting of $T_X$. Nevertheless $f$ is not a locally constant fibration.  
(In fact, $f$ is not even semistable.)
\end{itemize}
\end{rem}

\begin{proof}[Proof of {\hyperref[thm-main-term]{Theorem \ref*{thm-main-term}}}]

The proof  can be divided into three steps:

\setcounter{step}{0}
\begin{step}[{\bf Holomorphicity of MRC fibrations}]
A key observation in this step is that $\calF$ and $\calG$ are weakly regular foliations (cf.\,\cite[Definition 5.4, Lemma 5.8]{Druel18b}) and $\calF \subset T_{X}$ is an algebraically integrable foliation. Therefore we can apply \cite[Theorem 4.6]{DGP20} to confirm that $\calF$ is induced by an equidimensional fiber space: 

\begin{lemm}
\label{lemma_F-canonical-sing}
Let the conditions be the same as those in {\hyperref[thm-main-term]{Theorem \ref*{thm-main-term}}}. 
Then, there exists a holomorphic MRC fibration $f \colon X\to Y$ to  a normal projective variety $Y$
$($which may differ from the original $Y$ chosen in {\hyperref[setting]{Setting \ref*{setting}}}$)$
such that 
\begin{enumerate}
\item[$\bullet$] $f \colon X\to Y$ has equidimensional fibers;
\item[$\bullet$] $\calF$ coincides with the foliation induced by $f \colon X\to Y$.
\end{enumerate}
\end{lemm}

Since $K_{\calG}\sim_{\QQ}0$, we have $K_Y\sim_{\QQ}0$.  
Thus we establish the first conclusion of {\hyperref[thm-main-term]{Theorem \ref*{thm-main-term}}}. 
It remains to show the second conclusion.
\end{step}

\begin{step}[{\bf Semi-stable reduction in codimension one and local constancy of the pullback fibration}]
From the previous step, we have a holomorphic MRC fibration $f\colon X\to Y$. 
We now intend to apply {\hyperref[thm_rel-anti-nef]{Theorem \ref*{thm_rel-anti-nef}}} to show that $f$ is a locally constant fibration 
(up to a quasi-\'etale cover). To this end, we take a resolution of $Y$ and consider the pullback fibration over it, 
but this new fibration is not 
necessarily semi-stable in codimension one. 
This means that the relative anti-canonical divisor is not necessarily nef, 
which prevents us from applying {\hyperref[thm_rel-anti-nef]{Theorem \ref*{thm_rel-anti-nef}}}. 
To overcome this obstruction, we need to take a further finite cover as we precise in the sequel. 

The morphism $f \colon X \to Y$ is semi-stable in codimension one by {\hyperref[prop-2.2]{Proposition \ref*{prop-2.2}(e)}}; 
hence the ramification divisor of $f$ is zero (cf.\,\cite[Definition 2.16]{CKT16}). 
Then, from \cite[Lemma 2.31]{CKT16} and $K_{\calG}\sim_{\QQ}0$, we have 
$$
K_{X/Y}\sim K_{\calF}\sim_{\QQ} K_X,
$$ 
which indicates that $K_Y\sim_{\QQ}0$. 
Meanwhile, since $(X,\Delta)$ is a $\QQ$-factorial terminal pair, 
we see that $(\calF,\Delta)$ has canonical singularities 
by writing 
\[
-K_{\calF}\sim_{\QQ}-(K_X+\Delta)+\Delta
\]
and by \cite[Proposition 5.5]{Druel17}. 
Then, we consider the following diagram:
\begin{center}
\begin{tikzpicture}[scale=2.5]
\node (A) at (0,0) {$Y$.};
\node (B) at (0,1) {$X$};
\node (A1) at (-1,0) {$Y_1$};
\node (B1) at (-1,1) {$X_1$};
\node (A2) at (-2,0) {$Y_2$};
\node (B2) at (-2,1) {$X_2$};
\path[->,font=\scriptsize,>=angle 90]
(B) edge node[right]{$f$} (A)
(B1) edge node[left]{$f_1$} (A1)
(B1) edge node[above]{$\mu_X$} (B)
(A1) edge node[below]{$\mu$} (A)
(B2) edge node[above]{$\sigma_X$} (B1)
(B2) edge node[left]{$f_2$} (A2)
(A2) edge node[below]{$\sigma$} (A1);
\end{tikzpicture}  
\end{center}
Here $\mu \colon Y_1\to Y$ is a resolution of singularities of $Y$ and $\sigma \colon Y_2\to Y_1$ is a finite cover (of Kawamata) that renders $f$ semi-stable in codimension one, and $X_1$ (resp.\,$X_2$) is the normalization of the fiber product $X\times_YY_1$ (resp.\,$X\times_YY_2$). 
Note that, since $f$ is equidimensional, the fiber product $X\times_YY_i$ has only one irreducible component for $i=1,2$.  
Then, by \cite[Proposition 2.4.20]{Wang-thesis}, we see that $\calF_1:=\mu_X^{-1}\calF=T_{X_1/Y_1}$ 
and $K_{\calF_1}\sim_{\QQ}\mu_X^\ast K_{\calF}$    
since $\calF$ has canonical singularities. Moreover, the branch locus of $\sigma$ is invariant under $\calF_1$ (cf.\,\cite[\S 3.4]{Druel18b}), and hence we obtain 
\[
K_{\calF_2}\sim_{\QQ}\sigma_X^\ast K_{\calF_1}\sim_{\QQ}(\mu_X\circ\sigma_X)^\ast K_{\calF}
\]
by \cite[Lemma 3.4]{Druel18b} 
(cf.\,\cite[Proof of Lemma 4.3]{Druel18b}). 
Meanwhile, since $f_2$ is semi-stable in codimension one, we have  $K_{X_2/Y_2}\sim K_{\calF_2}$. 
Set $\Delta_2:=(\mu_X\circ\sigma_X)^\ast\Delta$ (noting that $X$ is $\QQ$-factorial). 
Then, by \cite[Proposition 5.6]{Druel17}, we see that $(X_2, \Delta_2)$ is canonical, 
and thus conclude that $f_2$ is a locally constant fibration by {\hyperref[thm_rel-anti-nef]{Theorem \ref*{thm_rel-anti-nef}}}.
\end{step}

\begin{step}[{\bf Conclusion}]
In this last step, we will deduce the latter conclusion of {\hyperref[thm-main-term]{Theorem \ref*{thm-main-term}}} 
from the local constancy of $f_2$. By taking the Stein factorization of $\mu\circ\sigma \colon Y_2\to Y$, 
we get a finite cover $\gamma \colon Y'\to Y$, which is ramified along the non-semi-stable locus of $f$, 
and thus it is quasi-\'etale. Let  $X'$ be the fiber product $X\times_Y Y'$ with the induced morphisms 
$f' \colon X'\to Y'$ and $\gamma_X \colon X'\to X$. Then, we have the following commutative diagram:
\begin{center}
\begin{tikzpicture}[scale=2.5]
\node (A) at (0,0) {$Y$.};
\node (B) at (0,1) {$X$};
\node (A1) at (-1,0) {$Y'$};
\node (B1) at (-1,1) {$X'$};
\node (A2) at (-2,0) {$Y_2$};
\node (B2) at (-2,1) {$X_2$};
\node (C) at (-0.5,0.5) {$\square$};
\path[->,font=\scriptsize,>=angle 90]
(B) edge node[right]{$f$} (A)
(B1) edge node[left]{$f'$} (A1)
(B1) edge node[above]{$\gamma_X$} (B)
(A1) edge node[below]{$\gamma$} (A)
(B2) edge node[above]{$\rho_X$} (B1)
(B2) edge node[left]{$f_2$} (A2)
(A2) edge node[below]{$\rho$} (A1);
\end{tikzpicture}  
\end{center}
Here, the morphisms $\rho$ and $\rho_X$ are birational by construction. Since $\gamma$ is quasi-\'etale, so is $\gamma_X$.  
Moreover, set $\Delta':=\gamma_X^\ast\Delta$, then $(X',\Delta')$ is terminal and $-(K_{X'}+\Delta')$ is nef. Moreover, since $Y$ is non-uniruled, so is $Y'$. Hence $f'$ coincides with the MRC fibration of $X'$. 
Since $f'$ is a locally trivial fibration by {\hyperref[lemme_lcf-birational]{Lemma \ref*{lemme_lcf-birational}}}, 
the left square in the diagram must be Cartesian. 
Then, by {\hyperref[lemma_Cartesian-lcf]{Lemma \ref*{lemma_Cartesian-lcf}}}, 
the fibration $f'$ is locally constant  with respect to $(X',\Delta')$. 
To conclude, we use the fact that $X$ is maximally quasi-\'etale. 
This indicates that the cover $\gamma_X$ is in fact \'etale, and thus so is $\gamma$. 
As a consequence, the fibration $f$ is a locally constant fibration with respect to $(X,\Delta)$. 
The proof of {\hyperref[thm-main-term]{Theorem \ref*{thm-main-term}}} is thus completed.  
\end{step}
\end{proof}

\subsection{Case of klt pairs with nef anti-log canonical divisor}
\label{ss_MRC_klt}
In this subsection, we deduce {\hyperref[thm-main]{Theorem \ref*{thm-main}}} 
by considering the case treated in {\hyperref[{ss_MRC_terminal}]{Subsection \ref*{ss_MRC_terminal}}}.

\begin{proof}[Proof of {\hyperref[thm-main]{Theorem \ref*{thm-main}}}]

Let $(X,\Delta)$ be a projective klt pair with nef anti-log canonical divisor $-(K_{X}+\Delta)$. 
Up to replacing $(X, \Delta)$ with a maximally quasi-\'etale cover of $X$ with the boundary divisor defined by pullback, 
we may assume that $X$ is maximally quasi-\'etale (see {\hyperref[thm-mqe]{Theorem \ref*{thm-mqe}(1)}}).  
Take a $\QQ$-factorial terminal model $g\colon (X',\Delta')\to (X,\Delta)$ of $(X,\Delta)$. 
Then, the variety $X'$ is also maximally quasi-\'etale by {\hyperref[thm-mqe]{Theorem \ref*{thm-mqe}(2)}}. 
By {\hyperref[thm-main-term]{Theorem \ref*{thm-main-term}}}, the MRC fibration $f' \colon X'\to Y'$ of $X'$ is a locally constant fibration with respect to $(X',\Delta')$ with $Y'$ having terminal singularities and $K_{Y'}\sim_{\QQ}0$. 
Moreover, the splitting of $T_{X'}$ 
induces a splitting of the tangent sheaf of $X$: $T_X\simeq\calF\oplus\calG$ such that $\calF$ (resp.\,$\calG$) coincides with $\calF'$ (resp.\,$\calG'$) over $X\setminus  g(\Exceptional(g))$ via the isomorphism $X'\setminus \Exceptional(g)\simeq X\setminus  g(\Exceptional(g))$. 
Moreover since $K_{Y'}\sim_{\QQ}0$, we have $K_{\calG}\sim_{\QQ}0$.

This subsection aims to show that $f' \colon X' \to Y'$ induces 
an MRC fibration fibration $f \colon X\to Y$ satisfying the conclusions of {\hyperref[thm-main]{Theorem \ref*{thm-main}}}, 
namely, that  $f \colon X\to Y$ is a locally constant fibration with respect to the pair $(X, \Delta)$ and $Y$ is a projective klt variety with the numerically trivial $K_{Y}$. The proof is divided into two steps: 

\setcounter{step}{0}
\begin{step}[{\bf Holomorphicity of MRC fibrations}]
First, we prove that $f'$ gives rise to an MRC fibration $f\colon X\to Y$ of $X$ defined everywhere.

\begin{theo}
\label{prop_birational-local-const}
Let the conditions remain the same as above. Then, there exists a normal projective variety $Y$ with canonical singularities whose canonical divisor $K_Y\sim_{\QQ}0$, a birational morphism $g_Y \colon Y'\to Y$, and an equidimensional fiber space $f\colon X\to Y$ such that $f\circ g=g_Y\circ f'$ and $f$ coincides with the MRC fibration of $X$. 
\end{theo}

\begin{proof}[Proof of {\hyperref[prop_birational-local-const]{Theorem \ref*{prop_birational-local-const}}}]
The proof of this theorem makes use of {\hyperref[lemme_nef-big]{Lemma \ref*{lemme_nef-big}}}.
Let us take a very ample divisor $A$ on $X$ and set $B=g^\ast\!A$. 
Then $B$ is a nef and big divisor on $X'$. The fiber of $f'$ has vanishing irregularity since it has terminal (and thus rational) singularities and is rationally connected. Hence, from {\hyperref[lemme_nef-big]{Lemma \ref*{lemme_nef-big}}},  we can write 
$B=B_0+(f')^\ast\!B_{Y'}$, where $B_0$ is nef and $B_{Y'}$ is nef and big. By $K_{Y'}\sim_{\QQ}0$, the basepoint-free theorem \cite[\S 3.1, Theorem 3.3, p.\,75]{KM98} implies that $B_{Y'}$ is semi-ample. Up to multiplying  by a sufficiently large and divisible integer, the divisor $B_{Y'}$ induces a birational morphism $g_Y\colon Y'\to Y$ (cf.\,\cite[\S 2.1.B, Theorem 2.1.27, pp.\,129-131, Vol.I]{Laz04}) and there is a very ample line bundle $A_Y$ on $Y$ such that $B_{Y'}\simeq g_Y^\ast A_Y$. Up to replacing $Y$ with its normalization, we can assume that $Y$ is normal. 

We next prove that $g_Y\circ f'$ factorizes through $g\colon X' \to X$. 
To this end, let $C$ be a curve contracted by $g$. 
Then,  we have $g^\ast\!A\cdot C=B\cdot C=0$ by the choice of $C$. 
Since both $B_{Y'}$ and $B_0$ are nef, then we have
\[
(f')^\ast\!B_{Y'}\cdot C=B_0\cdot C=0,
\] 
which implies that $(g_Y\circ f')^\ast\!A_Y\cdot C=0$. Hence $C$ is contracted by $g_Y\circ f'$. By \cite[Lemma 1.15, pp.\,12-13]{Deb01}, 
there is a morphism $f  \colon X\to Y$ such that $g_Y\circ f'=f\circ g$. 
We claim that $f$ is equidimensional. To see this, let $X''$ be the normalization of $X\times_YY'\to Y'$. 
Then, by considering everything over the universal cover of $Y'$ and applying \cite[Lemma 4.2.3]{Wang-thesis}, 
we find that $X''\to Y'$ is locally trivial and in particular equidimensional, and thus so is $f$. 
Moreover, since $g_Y$ is birational, $Y$ is not uniruled.  
Hence $f$ must coincide with the MRC fibration of $X$ 
and $T_{X/Y}\simeq\calF$.

By {\hyperref[prop-2.2]{Proposition \ref*{prop-2.2}(e)}}, 
the morphism $f$ is semi-stable in codimension one. Hence, the ramification divisor of $f$ is zero, and  $K_{X/Y}\sim K_{\calF}\sim_{\QQ}K_X$. Therefore, we obtain that $K_Y$ is $\QQ$-Cartier and $K_Y\sim_{\QQ}0$. 
Consequently, we have $K_{Y'}\sim_{\QQ}g_Y^\ast K_Y$, but $Y'$ has terminal singularities, and thus $Y$ has canonical singularities.   
\end{proof}
\end{step}

\begin{step}[{\bf Conclusion}]
To prove that $f$ is a locally constant fibration, we intend to apply {\hyperref[thm_rel-anti-nef]{Theorem \ref*{thm_rel-anti-nef}}} and {\hyperref[lemma_Cartesian-lcf]{Lemma \ref*{lemma_Cartesian-lcf}}}. 
Nevertheless, in order to apply {\hyperref[thm_rel-anti-nef]{Theorem \ref*{thm_rel-anti-nef}}}, we need to consider the base change fibration over a resolution of singularities of $Y$, but we do not know whether this still satisfies the assumptions of {\hyperref[thm_rel-anti-nef]{Theorem \ref*{thm_rel-anti-nef}}} (i.e.,\,the klt condition and nefness).

To overcome these difficulties, we apply {\hyperref[lemme_lcf-birational]{Lemma \ref*{lemme_lcf-birational}}} (noting that $Y$ has canonical singularities, in particular klt singularities) to show that $f$ is a locally trivial fibration. Let $\mu \colon Y_1\to Y$ be a resolution of singularities of $Y$. Since $f$ is locally trivial, the base change morphism $f_1 \colon X_1\to Y_1$ is still locally trivial and $(X_1,\Delta_1)$ is still klt, where $X_1:=X\times_Y Y_1$ and $\Delta_1$ is the base change of $\Delta$. Note that $\Delta_1$ is locally trivial over $Y_1$ since $\Delta$ is locally trivial over $Y$. The relative anti-canonical divisor $-(K_{X_1/Y_1}+\Delta_1)$, which is equal to the pullback of $-(K_X+\Delta)$, is nef. 
We now can apply {\hyperref[thm_rel-anti-nef]{Theorem \ref*{thm_rel-anti-nef}}} to $f_1\colon (X_{1},\Delta_1) \to Y_{1}$, 
and conclude that $f_1$ is a locally constant fibration with respect to $(X_1,\Delta_1)$. 
Then, since $Y$ has canonical singularities, the same holds for $f$ and $(X,\Delta)$ by {\hyperref[lemma_Cartesian-lcf]{Lemma \ref*{lemma_Cartesian-lcf}}}. This shows that $f$ is a locally constant fibration with respect to $(X,\Delta)$. 
\end{step}
Our main result {\hyperref[thm-main]{Theorem \ref*{thm-main}}} has thus been  proved. 
\end{proof}

\bibliographystyle{alpha}
\bibliography{anti-nef}

\begin{thebibliography}{KMM92b}

\bibitem[ACG11]{ACG11}
Enrico Arbarello, Maurizio Cornalba, and Phillip Griffiths.
\newblock {\em {Geometry of Algebraic Curves, Vol.I\!I}}, volume 268 of {\em
  Grundlehren der mathematischen Wissenschaften}.
\newblock Sprinnger-Verlag, Berlin Heidelberg, 2011.

\bibitem[Amb05]{Amb05}
Florin Ambro.
\newblock The moduli b-divisor of an lc-trivial fibration.
\newblock {\em Compositio Mathematica}, 141(2):385--403, 2005.

\bibitem[BCHM10]{BCHM10}
Caucher Birkar, Paolo Cascini, Christopher~Derek Hacon, and James McKernan.
\newblock {Existence of Minimal Models for Varieties of log General Type}.
\newblock {\em Journal of the American Mathematical Society}, 23(2):405--468,
  2010.

\bibitem[BDPP13]{BDPP13}
S{\'e}bastien Boucksom, Jean-Pierre Demailly, Mihai P{\u a}un, and Thomas
  Peternell.
\newblock {The Pseudo-effective Cone of a Compact K{\"a}hler Manifold and
  Varieties of Negative Kodaira Dimension}.
\newblock {\em Journal of Algebraic Geometry}, 22(2):201--248, 2013.

\bibitem[Bea83]{Bea83}
Arnaud Beauville.
\newblock Vari\'et\'es k\"ahl\'eriennes dont la premi\`ere classe de chern est
  nulle.
\newblock {\em Journal of Differential Geometry}, 18(4):755--782, 1983.

\bibitem[Cam92]{Cam92}
Fr{\'e}d{\'e}ric Campana.
\newblock {Connexit{\'e} rationelle des vari{\'e}t{\'e}s de Fano}.
\newblock {\em Annales scientifiques de l'{\'E}cole normale sup{\'e}rieure},
  25(5):539--545, 1992.

\bibitem[Cam10]{Cam10}
Fr\'ed\'eric Campana.
\newblock {Remarks on an Example of K.Ueno}.
\newblock In Carel Faber, Gerard van~der Geer, and Eduard Looijenga, editors,
  {\em {Classification of Algebraic Varieties}}, volume~7 of {\em EMS Series of
  Congree Reports}, pages 115--121, Zurich, 2010. EMS Publishing House.

\bibitem[Cam16]{Cam16}
Fr{\'e}d{\'e}ric Campana.
\newblock {Orbifold slope-rational connectedness}.
\newblock Preprint \urlstyle{rm}\url{https://arxiv.org/abs/1607.07829}, 2016.

\bibitem[Cam20]{Cam20}
Fr{\'e}d{\'e}ric Campana.
\newblock {The Bogomolov-Beauville-Yau Decomposition for Klt Projective
  Varieties with Trivial First Chern Class -Without Tears-}.
\newblock Preprint \urlstyle{rm}\url{https://arxiv.org/abs/2004.08261}, 2020.
\newblock to appear in {\it Bulletin de la Soci{\'e}t{\'e} Math{\'e}matique de
  France}.

\bibitem[Cao13]{Cao13}
Junyan Cao.
\newblock {\em Th{\'e}or{\`e}mes d'annulation et th{\'e}or{\`e}mes de structure
  sur les vari{\'e}t{\'e}s k{\"a}hl{\'e}riennes compactes}.
\newblock PhD thesis, Universit{\'e} de Grenoble (ancienne Universit{\'e}
  Grenoble-Alpes), 2013.

\bibitem[{Cao}19]{Cao19}
Junyan {Cao}.
\newblock {Albanese Maps of Porjective Manifolds with Nef Anticanonical
  Bundles}.
\newblock {\em Annales scientifiques de l'{\'E}cole normale sup{\'e}rieure},
  52(5):1137--1154, 2019.

\bibitem[CC96]{CC96}
Jeff Cheeger and Tobias~Holck Colding.
\newblock {Lower Bounds on Ricci Curvature and the Almost Rigidity of Warped
  Products}.
\newblock {\em Annals of Mathematics}, 144(1):189--237, 1996.

\bibitem[CCM19]{CCM19}
Fr{\'e}d{\'e}ric Campana, Junyan Cao, and Shin-ichi Matsumura.
\newblock {Projective Klt Pairs with Nef Anti-canonical Divisor}.
\newblock Preprint \urlstyle{rm} \url{https://arxiv.org/abs/1910.06471}, 2019.
\newblock to appear in {\it Algebraic Geometry}.

\bibitem[CH19]{CH19}
Junyan Cao and Andreas {H{\"o}ring}.
\newblock {A Decomposition Theorem for Projective Manifolds with Nef
  Anticanonical Bundle}.
\newblock {\em Journal of Algebraic Geometry}, 28(3):567--597, 2019.

\bibitem[CKT16]{CKT16}
Beno{\^i}t Claudon, Stefan Kebekus, and Behrouz Taji.
\newblock {Generic Positivity and Applications to Hyperbolicity of Moduli
  Spaces}.
\newblock Preprint \urlstyle{rm}
  \url{https://hal.archives-ouvertes.fr/hal-01389525/document}, 2016.
\newblock to appear as a chapter of a book about complex hyperbolicity.

\bibitem[CLN85]{CLN85}
C\'esar Camacho and Alcides Lins~Neto.
\newblock {\em {Geometric Theory of Foliations}}.
\newblock Birkh{\"a}user, Inc., Basel, 1985.

\bibitem[CP91]{CP91}
Fr{\'e}d{\'e}ric Campana and Thomas Peternell.
\newblock {Projective Manifolds whose Tangent Bundles are Numerically
  Effective}.
\newblock {\em Mathematische Annalen}, 289(1):169--187, 1991.

\bibitem[CP19]{CP19}
Fr\'{e}d\'{e}ric Campana and Mihai P\u{a}un.
\newblock {Foliations with Positive Slopes and Birational Stability of Orbifold
  Cotangent Bundles}.
\newblock {\em {Publications math\'{e}matiques de I.H.\'E.S.}}, 129:1--49,
  2019.

\bibitem[{Deb}01]{Deb01}
Olivier {Debarre}.
\newblock {\em {Higher-Dimensional Algebraic Geometry}}.
\newblock Universitext. Springer-Verlag, New York, NY, 2001.

\bibitem[Dem01]{Dem01}
Jean-Pierre Demailly.
\newblock {Multiplier Ideal Sheaves and Analytic Methods in Algebraic
  Geometry}.
\newblock In Jean-Pierre Demailly, Lothar G{\"{o}}ttsche, and Robert
  Lazarsfeld, editors, {\em {Vanishing Theorems and Effective Results in
  Algebraic Geometry}}, volume~6 of {\em ICTP Lecture Notes Series}, pages
  1--148, Trieste, 2001. The Abdus Salam International Centre for Theoretical
  Physics.

\bibitem[{Dem}10]{Dem10}
Jean-Pierre {Demailly}.
\newblock {\em {Analytic Methods in Algebraic Geometry}}, volume~1 of {\em
  Surveys of Modern Mathematics}.
\newblock Higher Education Press; International Press, Beijing; Sommerville,
  2010.

\bibitem[Den17]{Deng17b}
Ya~Deng.
\newblock {\em {Le corps d'Okounkov g\'e\'eralis\'e et des probl\`emes li\'es
  \`a l'hyperbolicit\'e et l'image directe}}.
\newblock PhD thesis, Universit{\'e} Grenoble-Alpes \& University of Science
  and Technology of China, 2017.

\bibitem[DGP20]{DGP20}
St{\'e}phane Druel, Henri Guenancia, and Mihai P{\u a}un.
\newblock {A Decomposition Theorem for $\mathbb{Q}$-Fano K\"ahler-Einstein
  Varieties}.
\newblock Preprint \urlstyle{rm}\url{https://arxiv.org/abs/2008.05352}, 2020.

\bibitem[DLB20]{DLB20}
St{\'e}phane Druel and Frederico Lo~Bianco.
\newblock {Numerical Characterization of Some Toric Fiber Bundles}.
\newblock Preprint \urlstyle{rm}\url{https://arxiv.org/abs/2003.13818}, 2020.

\bibitem[DPS94]{DPS94}
Jean-Pierre Demailly, Thomas Peternell, and Michael Schneider.
\newblock {Compact Complex Manifolds with Numerically Effective Tangent
  Bundles}.
\newblock {\em Journal of Algebraic Geomery}, 3(2):295--345, 1994.

\bibitem[Dru17]{Druel17}
St{\'e}phane Druel.
\newblock {On Foliations with Nef Anti-canonical Bundle}.
\newblock {\em Transactions of the American Mathematical Society},
  369(11):7765--7787, 2017.

\bibitem[Dru18a]{Druel18a}
St{\'e}phane Druel.
\newblock {A Decomposition Theorem for Singular Spaces with Trivial Canonical
  Class of Dimension at most five}.
\newblock {\em Inventiones mathematicae}, 211(1):245--296, 2018.

\bibitem[Dru18b]{Druel18b}
St{\'e}phane Druel.
\newblock {Codimension one Foliations with Numerically Trivial Canonical Class
  on Singular Spaces}.
\newblock Preprint \urlstyle{rm}\url{https://arxiv.org/abs/1809.06905}, 2018.
\newblock To appear in {\it Duke Mathemaical Journal}.

\bibitem[EIM20]{EIM20}
Sho Ejiri, Masataka Iwai, and Shin-ichi Matsumura.
\newblock {On Asympotic Base Loci of Relative Anti-canonical Divisors of
  Algebraic Fibre Spaces}.
\newblock \urlstyle{rm}\url{https://arxiv.org/abs/2005.04566}, 2020.

\bibitem[Fuj16]{Fuj16}
Osamu Fujino.
\newblock {Direct Images of Relative Pluricanonical Bundles}.
\newblock {\em Algebraic Geometry}, 3(1):50--62, 2016.

\bibitem[GGK19]{GGK19}
Daniel Greb, Henri Guenancia, and Stefan Kebekus.
\newblock {Klt Varieties with Trivial Canonical Class: Holonomy, Differential
  Forms and Fundamental Groups}.
\newblock {\em Geometry \& Topology}, 23(4):2051--2124, 2019.

\bibitem[GHS03]{GHS03}
Tom Graber, Joe Harris, and Jason Starr.
\newblock {Families of Rationally Connected Varieties}.
\newblock {\em Journal of the American Mathematical Society}, 16(1):57--67,
  2003.

\bibitem[GKKP11]{GKKP11}
Daniel Greb, Stefan Kebekus, S{\'a}ndor~J{\'o}zsef Kov{\'a}cs, and Thomas
  Peternell.
\newblock {Differential Forms on Log Canonical Spaces}.
\newblock {\em Publications math{\'e}matiques de l'I.H.{\'E}.S.},
  114(1):87--169, 2011.

\bibitem[GKP16a]{GKP16b}
Daniel Greb, Stefan Kebekus, and Thomas Peternell.
\newblock {\'Etale Fundamental Groups of Kawamata Log Terminal Spaces, Flat
  Sheaves and Quotients of Abelian Varieties}.
\newblock {\em Duke Mathematical Journal}, 165(10):1965--2004, 2016.

\bibitem[GKP16b]{GKP16a}
Daniel Greb, Stefan Kebekus, and Thomas Peternell.
\newblock {Singular Spaces with Trivial Canonical Class}.
\newblock In J{\'a}nos Koll{\'{a}}r, Osamu Fujino, Shigeru Mukai, and Noboru
  Nakayama, editors, {\em {Minimal Models and Extremal Rays}}, volume~70 of
  {\em Advanced Studies in Pure Mathematics}, pages 67--113, Kyoto, 2016.
  Mathematical Society of Japan.

\bibitem[GM89]{GM89}
Xavier G\'omez-Mont.
\newblock {Integrals for Holomorphic Foliations with Singularities Having all
  Leaves Compact}.
\newblock {\em Annales de l'Institut Fourier}, 39(2):451--458, 1989.

\bibitem[GR56]{GR56}
Hans {Grauert} and Reinhold {Remmert}.
\newblock {Plurisubharmonische Funktionen in komplexen R\"aumen}.
\newblock {\em Mathematische Zeitschrift}, 65:175--194, 1956.

\bibitem[{Har}77]{Har77}
Robin {Hartshorne}.
\newblock {\em {Algebraic Geometry}}, volume~52 of {\em Graduate Texts in
  Mathematics}.
\newblock Springer-Verlag, New York, NY, 1977.

\bibitem[{Har}80]{Har80}
Robin {Hartshorne}.
\newblock {Stable Reflexive Sheaves}.
\newblock {\em Mathematische Annalen}, 254(2):121--176, 1980.

\bibitem[HIM19]{HIM19}
Genki Hosono, Masataka Iwai, and Shin-ichi Matsumura.
\newblock On projective manifolds with pseudo-effective tangent bundle.
\newblock Preprint \urlstyle{rm}\url{https://arxiv.org/abs/1908.06421}, 2019.
\newblock to appear in {\it Journal of the Institute of Mathematics of
  Jussieu}.

\bibitem[HM07]{HM07}
Chritopher~Derek Hacon and James McKernan.
\newblock {On Shokurov's Rational Connectedness Conjecture}.
\newblock {\em Duke Mathematical Journal}, 138(1):119--136, 2007.

\bibitem[H{\"o}r07]{Hor07}
Andreas H{\"o}ring.
\newblock {Uniruled Varieties with Split Tangent Bundles}.
\newblock {\em Mathematische Zeitschrift}, 256(3):465--479, 2007.

\bibitem[HP19]{HP19}
Andreas {H{\"o}ring} and Thomas Peternell.
\newblock {Algebraic Integrability of Foliations with Numerically Trivial
  Canonical Bundle}.
\newblock {\em Inventiones Mathematicae}, 216(2):395--419, 2019.

\bibitem[HPS18]{HPS18}
Christopher~Derek {Hacon}, Mihnea {Popa}, and Christian {Schnell}.
\newblock {Algebraic Fibre Spaces over Abelian Varieties: Around a Recent
  Theorem by Cao and P\u aun}.
\newblock In Nero {Budur}, Tommaso {de Fernex}, Roi {Docampo}, and Kevin
  {Tucker}, editors, {\em Local and Global Methods in Algebraic Geometry},
  volume 712 of {\em Contemporary Mathematics}, pages 143--195, Providence, RI,
  2018. American Mathematical Society.

\bibitem[HSW81]{HSW81}
Alan Howard, Brian Smyth, and Hung-Hsi Wu.
\newblock {On Compact K{\"a}hler Manifolds of Nonnegative Bisectional
  Curvature, I}.
\newblock {\em Acta Mathematica}, 147:51--56, 1981.

\bibitem[HW20]{HW20}
Gordon Heier and Bun Wong.
\newblock {On Projective K\"ahler Manifolds of Partially Positive Curvature and
  Rational Connectedness}.
\newblock {\em Documenta Mathematica}, 25:219--238, 2020.

\bibitem[KKLV89]{KKLV89}
Friedrich Knop, Hanspeter Kraft, Domingo Luna, and Thierry Vust.
\newblock {Local Properties of Algebraic Group Actions}.
\newblock In Hanspeter Kraft, Peter Slowdowy, and Tonny~Albert Springer,
  editors, {\em {Algebraische Transformationsgruppen und Invariantentheorie}},
  volume~13 of {\em DMV Seminar book series}, pages 63--75, Basel, 1989.
  Birkh{\"a}user.

\bibitem[KL09]{KL09}
J\'anor Koll\'ar and Michael Larsen.
\newblock {Quotients of Calabi-Yau Varieties}.
\newblock In Yuri Tschinkel and Yuri Zarhin, editors, {\em {Algebra,
  Arithmetic, and Geometry, Volume I\!I: In honor of Yu.I.Manin}}, volume 270
  of {\em Progress in Mathematics}, pages 179--211, Boston, 2009.
  Birkh{\"a}user.

\bibitem[KM98]{KM98}
J{\'a}nos {Koll{\'a}r} and Shigefumi {Mori}.
\newblock {\em {Birational Geometry of Algebraic Varieties}}, volume 134 of
  {\em Cambridge Tracts in Mathematics}.
\newblock Cambridge University Press, Cambridge, 1998.

\bibitem[KMM92a]{KoMM92}
J{\'a}nos Koll{\'a}r, Yoichi Miyaoka, and Shigefumi Mori.
\newblock {Rational Connectedness and Boundedness of Fano Manifolds}.
\newblock {\em Journal of Differential Geometry}, 36(3):765--779, 1992.

\bibitem[KMM92b]{KoMM92a}
J\'anos Koll\'ar, Yoichi Miyaoka, and Shigefumi Mori.
\newblock {Rationally Connected Varieties}.
\newblock {\em Journal of Algebraic Geometry}, 1:429--448, 1992.

\bibitem[{Kob}87]{Kob87}
Shoshichi {Kobayashi}.
\newblock {\em {Differential Geometry of Complex Vector Bundles}}.
\newblock Princeton Legacy Library. Princeton University Press, Princeton, NJ,
  1987.

\bibitem[Kol07]{Kollar07}
J{\'a}nos Koll{\'a}r.
\newblock {\em {Lectures on Resolution of Singularities}}, volume 166 of {\em
  Annals of Mathematics Studies}.
\newblock Princeton University Press, Princeton Oxford, 2007.
\newblock Corrected Second Printing 1999.

\bibitem[{Laz}04]{Laz04}
Robert {Lazarsfeld}.
\newblock {\em {Positivity in Algebraic Geometry I \& II}}, volume 48 \& 49 of
  {\em Ergebnisse der Mathematik und ihrer Grenzgebiete. 3. Folge}.
\newblock Springer-Verlag, Berlin Heidelberg, 2004.

\bibitem[Mat18]{Mat18c}
Shin-ichi Matsumura.
\newblock {On Projective Manifolds with Semi-positive Holomorphic Sectional
  Curvature}.
\newblock Preprint \urlstyle{rm}\url{https://arxiv.org/abs/1811.04182}, 2018.
\newblock to appear in {\it American Journal of Mathematics}.

\bibitem[Mok88]{Mok88}
Ngaiming Mok.
\newblock {The Uniformisation Theorem for Compact K{\"a}hler Manifolds of
  Nonnegative Holomorphic Bisectional Curvature}.
\newblock {\em Journal of Differential Geometry}, 27(2):179--214, 1988.

\bibitem[Mor79]{Mori79}
Shigefumi Mori.
\newblock {Projective Manifolds with Ample Tangent Bundles}.
\newblock {\em Annals of Mathematics}, 110(3):593--606, 1979.

\bibitem[Mou97]{Mou97}
Christophe Mourougane.
\newblock {Images Directes de Fibr{\'e}s en Droites Adjoints}.
\newblock {\em Publications of the RIMS Kyoto University}, 33(6):893--916,
  1997.

\bibitem[MZ86]{MZ86}
Ngaming Mok and Jia~Qing Zhong.
\newblock {Curvature Characterization of Compact Hermitian Symmetric Spaces}.
\newblock {\em Journal of Differential Geometry}, 23(1):15--67, 1986.

\bibitem[P{\u a}u97]{Paun97}
Mihai P{\u a}un.
\newblock {Sur le Groupe Fondamental des Vari{\'e}t{\'e}s k{\"a}hl{\'e}riennes
  Compactes {\`a} Classe de Ricci Num{\'e}riquement effective}.
\newblock {\em Comptes Rendus de l'Acad{\'e}mie des Sciences},
  324(11):1249--1254, 1997.

\bibitem[P{\u{a}}u17]{Paun17}
Mihai P{\u{a}}un.
\newblock {Relative Adjoint Transcendental Classes and Albanese Map of Compact
  K{\"a}hler Manifolds with Nef Ricci Curvature}.
\newblock In Keiji Oguiso, Caucher Birkar, Shihoko Ishii, and Shigeharu
  Takayama, editors, {\em {Higher Dimensional Algebraic Geometry: In honour of
  Professor Yujiro Kawamata's Sixtieth Birthday}}, volume~47 of {\em Advanced
  Studies in Pure Mathematics}, pages 335--356, Tokyo, 2017. Mathematical
  Society of Japan.

\bibitem[P{\u{a}}u18]{Pau16}
Mihai P{\u{a}}un.
\newblock {Singular Hermitian Metrics and Positivity of Direct Images of
  Pluricanonical Bundles}.
\newblock In Tommaso de~Fernex, Brendan Hassett, Mircea Musta{\c{t}}{\u{a}},
  Martin Olsson, Mihnea Popa, and Richard Thomas, editors, {\em {Algebraic
  Geometry: Salt Lake City 2015, Part 1}}, volume~97 of {\em Proceedings of
  Symposia in Pure Mathematics}, pages 519--554, Providence, RI, 2018. American
  Mathematical Society.

\bibitem[PS14]{PS14}
Mihnea Popa and Christian Schnell.
\newblock {On Direct Images of Pluricanonical Bundles}.
\newblock {\em Algebra \& Number Theory}, 8(9):2273--2295, 2014.

\bibitem[PT18]{PT18}
Mihai P{\u{a}}un and Shigeharu Takayama.
\newblock {Positivity of Twisted Relative Pluricanonical Bundles and Their
  Direct Images}.
\newblock {\em Journal of Algebraic Geometry}, 27(2):211--272, 2018.

\bibitem[Rau15]{Raufi15}
Hossein Raufi.
\newblock {Singular Hermitian Metrics on Holomorphic Vector Bundles}.
\newblock {\em Arkiv f{\"u}r Mathematik}, 53(2):359--382, 2015.

\bibitem[Sim92]{Sim92}
Carlos~Tschudi Simpson.
\newblock {Higgs Bundles and Local Systems}.
\newblock {\em Publications mathm{\'e}matiques de l'I.H.{\'E}.S}, 75(1):5--95,
  1992.

\bibitem[SY80]{SY80}
Yum-Tong Siu and Shing-Tung Yau.
\newblock {Compact K{\"a}hler Manifolds of Positive Bisectional Curvature}.
\newblock {\em Inventiones mathematicae}, 59(2):189--204, 1980.

\bibitem[Tak03]{Tak03}
Shigeharu Takayama.
\newblock {Local Simple Connectedness of Resolution of Log-Terminal
  Singularities}.
\newblock {\em International Journal of Mathematics}, 14(8):825--836, 2003.

\bibitem[Wan20]{Wang-thesis}
Juanyong Wang.
\newblock {\em {Positivity of direct images and projective varieties with
  nonnegative curvature}}.
\newblock PhD thesis, {Institut Polytechnique de Paris}, Aug 2020.
\newblock \urlstyle{rm}\url{https://tel.archives-ouvertes.fr/tel-02982921}.

\bibitem[Wan22]{Wang20}
Juanyong Wang.
\newblock {Structure of Projective Varieties with Nef Anticanonical divisor:
  the Case of Log Terminal Singularities}.
\newblock {\em Mathematische Annalen}, 384(1-2):47--100, 2022.

\bibitem[Wu20]{Wu20}
Xiaojun Wu.
\newblock {Pseudo-effective and Numerically Flat Reflexive Sheaves}.
\newblock Preprint \urlstyle{rm}\url{https://arxiv.org/abs/2004.14676}, 2020.

\bibitem[Yan16]{Yang16}
Xiaokui Yang.
\newblock {Hermitian Manifolds with Semi-positive Holomorphic Sectional
  Curvature}.
\newblock {\em Mathematical Research Letters}, 23(3):939--952, 2016.

\bibitem[Yan18]{Yang18}
Xiaokui Yang.
\newblock {RC-positivity, Rational Connectedness and Yau’s Conjecture}.
\newblock {\em Cambridge Journal of Mathematics}, 6(2):183--212, 2018.

\bibitem[Zha96]{Zhang96}
Qi~Zhang.
\newblock {On Projective Manifolds with Nef Anticanonical Divisors}.
\newblock {\em Journal f{\"u}r die reine und angewandte Mathematik},
  478(3):57--60, 1996.

\bibitem[Zha05]{Zhang05}
Qi~Zhang.
\newblock {On Projective Varieties with Nef Anticanonical Divisors}.
\newblock {\em Mathematische Annalen}, 332(3):697--703, 2005.

\bibitem[Zha06]{Zhang06}
Qi~Zhang.
\newblock {Rational Connectedness of log $\QQ$-Fano Varieties}.
\newblock {\em Journal f{\"u}r die reine und angewandte Mathematik},
  2006(590):131--142, 2006.

\end{thebibliography}

\end{document}